\documentclass[11pt]{amsart}
\usepackage[T1]{fontenc}
\usepackage{newtxtext}

\usepackage[a4paper,margin=27mm,headheight=14pt]{geometry}
\usepackage{amsmath,amssymb,amsthm,mathtools}
\usepackage{newtxmath}
\usepackage{microtype}
\usepackage{needspace}
\usepackage[unicode=true,colorlinks=true,linkcolor=black,citecolor=black,urlcolor=black]{hyperref}
\newcommand{\R}{\mathbb R}
\newcommand{\dd}{\,\mathrm d}

\newcommand{\intr}{\operatorname{int}}
\newcommand{\dist}{\operatorname{dist}}
\newcommand{\rec}{\operatorname{rec}}

\newcommand{\ip}[2]{\langle #1,#2\rangle}
\numberwithin{equation}{section}
\allowdisplaybreaks[2]
\date{}

\newtheorem{theorem}{Theorem}[section]
\newtheorem{lemma}[theorem]{Lemma}
\newtheorem{proposition}[theorem]{Proposition}
\newtheorem{corollary}[theorem]{Corollary}
\theoremstyle{remark}
\newtheorem{remark}[theorem]{Remark}

\title[Ehrhard's inequality and sharp lower deviations]{From Cauchy to Gaussian: Sharp Lower Deviations and \\Ehrhard-Type Inequality}
\makeatletter
\@ifundefined{subjclassname@2020}{\@namedef{subjclassname@2020}{\textup{2020} Mathematics Subject Classification}}{}
\makeatother
\author[Y. Huang]{Yong Huang}
\address{Institute of Mathematics,
	Hunan University,  Changsha,  410082,  China}
\email{
	huangyong@hnu.edu.cn}

\author[S. Tang]{Shengyu Tang}
\address{Institute of Mathematics,
	Hunan University,  Changsha,  410082,  China}
\email{
	tsy@hnu.edu.cn}

\author[S. Wang]{Sinan Wang}
\address{Institute of Mathematics,
	Hunan University,  Changsha,  410082,  China}
\email{
	wangsinan@hnu.edu.cn}
\begin{document}
	\begin{abstract}
		We establish dimension-free, pointwise optimal lower-deviation bounds for convex functions under spherically symmetric Cauchy-type measures. These bounds use either mean–variance normalization or moments about the median, the latter covering the Cauchy case where the mean and variance are unavailable. By approximation with a parametrized family the Cauchy-type measures, our results recover the Gaussian setting of Paouris and Valettas \cite{PaourisValettas18} and strengthen their variance-sensitive estimate to a sharp bound, with a complete equality characterization. This goes beyond the partial improvements obtained in subsequent work.
		
		We also prove an Ehrhard-type inequality for arbitrary convex sets under these Cauchy-type measures in every dimension and characterize all equality cases. This resolves a conjecture posed by Byczkowski,
		Ma\l ecki and \.{Z}ak \cite{ByczkowskiMaleckiZak22}. The proof follows from a variational argument combined with the joint concavity of fiber rearrangement. The geometric inequality reduces distributions of convex observables to one-dimensional convex transforms, yielding the sharp lower-deviation bounds.
	\end{abstract}
	\maketitle
	\section{Introduction}
	
	The spherical Cauchy distribution is a fundamental model in
	heavy-tailed probability. Its characteristic function is
	$e^{-|\xi|}$, so it is the time-one distribution of the rotationally
	invariant symmetric $1$-stable process. Its density is also the
	Poisson kernel of the upper half-space and describes the boundary
	hitting location of Brownian motion started at height one.
	These representations connect the Cauchy measure with stable
	processes, harmonic measure, and the nonlocal semigroup generated
	by $-(-\Delta)^{1/2}$, studied by Caffarelli and Silvestre \cite{CaffarelliSilvestre07}. These connections and their geometric
	consequences were studied systematically by
	Ba\~nuelos and Kulczycki \cite{BanuelosKulczycki04}.
	
	In high-dimensional probability, Cauchy-type measures also form a fundamental class of heavy-tailed convex measures. Borell's theory of concave measures places Gaussian distributions and distributions with polynomial tails within a common framework, relating the interpolation properties of sets to the integrability of functions \cite{Borell74}. Bobkov subsequently developed large-deviation, small-deviation, and isoperimetric estimates for heavy-tailed convex measures \cite{Bobkov07}. Bobkov and Ledoux established weighted functional inequalities connecting moments and deviation probabilities for Cauchy-type measures \cite{BobkovLedoux09}. In the heavy-tailed setting, the tail parameter remains essential, and the appropriate normalization may differ substantially from the variance scale familiar in the Gaussian setting.
	
	This paper is concerned with lower deviations of convex observables. Natural examples include norms, maximum of linear forms, distances to convex sets, and convex loss functions. Our aim is to derive the optimal lower-tail bound in terms of the reference distribution and the dispersion of the observable itself. Such a normalization reflects the actual fluctuations of the random variable and allows for convex functions that are not globally Lipschitz.
	
	This problem has important precedents in Gaussian probability. Classical Gaussian concentration asserts that, if
	$G_n\sim N(0,I_n)$ and $f:\mathbb R^n\to\mathbb R$ is $L$-Lipschitz,
	then
	\begin{equation}\label{intro-classical-concentration}
		\mathbb P\!\left(
		|f(G_n)-\mathbb Ef(G_n)|\ge tL
		\right)
		\le 2e^{-t^2/2},
		\qquad t>0.
	\end{equation}
	The geometric foundation is the Gaussian isoperimetric inequality
	of Sudakov--Tsirelson and Borell
	\cite{SudakovTsirelson74,BorellGaussian75}; see also \cite{Ledoux01}.
	The Lipschitz constant, however, may substantially overestimate the
	standard deviation. This raises the question of whether additional
	convexity permits deviation estimates on the scale of the actual
	fluctuations.
	
	Paouris and Valettas \cite{PaourisValettas18} answered this question
	for lower deviations. In the following Gaussian comparisons, write
	$Y=f(G_n)$, where $f$ is convex and
	$0<\operatorname{Var}Y<\infty$. Their estimate gives
	\begin{equation}\label{pv-background}
		\mathbb P(Y\le\mathbb EY-t\sqrt{\operatorname{Var}Y})
		\le
		\frac12\exp\!\left(-\frac{\pi(t-1)^2}{1024}\right),
		\qquad t>1.
	\end{equation}
	No Lipschitz assumption is required. Their argument uses Ehrhard's
	inequality to convert the geometry of convex sublevel sets into
	concavity of a transformed distribution function.
	
	Valettas \cite[Theorem~2.15]{Valettas17} subsequently obtained a sharp
	estimate under normalization by a median and a positive first
	moment. If $m$ is a median of $Y$, then
	\begin{equation}\label{intro-valettas-positive}
		\mathbb P\!\left(
		Y<m-s\,\mathbb E(Y-m)_+
		\right)
		\le
		\Phi\!\left(-\frac{s}{\sqrt{2\pi}}\right),
		\qquad s>0,
	\end{equation}
	where $x_+=\max\{x,0\}$ and $\Phi$ is the standard Gaussian
	distribution function. The elementary median estimates
	\[
	\mathbb E(Y-m)_+\le\sqrt{\operatorname{Var}Y},
	\qquad
	m\ge\mathbb EY-\sqrt{\operatorname{Var}Y}
	\]
	yield
	\begin{equation}\label{intro-valettas-variance}
		\mathbb P(Y\le\mathbb EY-t\sqrt{\operatorname{Var}Y})
		\le
		\Phi\!\left(-\frac{t-1}{\sqrt{2\pi}}\right),
		\qquad t>1.
	\end{equation}
	For every $t>1$,
	\eqref{intro-valettas-variance} strictly improves the explicit
	bound \eqref{pv-background}. More recently, Chen \cite{Chen26} used a convexity property of
	Gaussian logarithmic moment generating functions to prove
	\begin{equation}\label{intro-chen}
		\mathbb P(Y\le\mathbb EY-t\sqrt{\operatorname{Var}Y})
		\le e^{-t^2/2},
		\qquad t>0.
	\end{equation}
	This improves the quadratic exponential rate to $1/2$ in the
	large-deviation regime. It is stronger than
	\eqref{intro-valettas-variance} for sufficiently large $t$, although
	the two bounds are not ordered in the same direction throughout
	$t>1$. Together, these results show that convexity permits
	variance-controlled Gaussian lower deviations without global Lipschitz regularity.
	
	Lower-deviation estimates also arise in more specific probabilistic problems. Paouris, Tikhomirov, and Valettas \cite{PaourisTikhomirovValettas22} studied small-ball probabilities for norms through parameters that involve variance, derivatives, and linear structure. Lopes and Yao \cite{LopesYao22} employed lower-tail estimates for Gaussian maxima in high-dimensional bootstrap methods. These applications illustrate the role played by lower deviations in the analysis of high-dimensional random structures, random maxima, and statistical approximation.
	
	We develop sharp lower-deviation comparisons for spherical
	Cauchy-type measures.
	For $\nu>0$, define
	\begin{equation}
		\label{student}
		\begin{cases}
			d\mu_{n,\nu}(x)&=\rho_{n,\nu}(x)\,dx,\\
			\rho_{n,\nu}(x)
			&=
			\frac{\Gamma((n+\nu)/2)}
			{\Gamma(\nu/2)(\nu\pi)^{n/2}}
			\left(1+\frac{|x|^2}{\nu}\right)^{-(n+\nu)/2}.
		\end{cases}
	\end{equation}
	For finite $\nu$ and $n\ge2$, the common random scale makes
	$\mu_{n,\nu}$ a non-product measure. Let $N(0,I_n)$ is the standard Gaussian distribution, and $\chi^2_\nu$ is the chi-square distribution with $\nu$ degree of freedom. If $G_n\sim N(0,I_n)$ and $W_\nu\sim\chi_\nu^2$ are independent,
	then
	\begin{equation}\label{mixture}
		\sqrt{\frac{\nu}{W_\nu}}\,G_n\sim\mu_{n,\nu}.
	\end{equation}
	The choice $\nu=1$ gives the rotationally invariant Cauchy measure,
	whereas $\nu\to\infty$ gives standard Gaussian measure.
	Write
	\begin{equation}\label{cdf}
		T_\nu(t)=\mu_{1,\nu}((-\infty,t]),
		\qquad
		\varphi_\nu=T_\nu'.
	\end{equation}
	For fixed $n$,
	\begin{equation}\label{intro-student-limit}
		\mu_{n,\nu}\longrightarrow\gamma_n
		\quad\text{in total variation},
		\qquad
		T_\nu^{-1}(p)\longrightarrow\Phi^{-1}(p),
		\quad 0<p<1.
	\end{equation}
	The family therefore provides a natural passage from Cauchy to
	Gaussian probability. Gaussian measure will arise as a limiting
	case of both our geometric and probability inequalities.
	
	Our first probability result determines the exact lower-tail curve under mean--variance normalization.
	For spherical Cauchy-type laws, a dimension-free statement must also retain the tail parameter. The reason is twofold: the distribution of a convex observable need not be a Cauchy distribution, and averaging conditional Gaussian estimates in \eqref{mixture} does not preserve the variance of the mixture. In view of this, instead of working only with the variance, we determine the whole distribution class directly. This turns the lower-deviation problem into a scalar extremal problem, and the resulting comparison identifies both the best probability bound and the extremal law. In addition, we also obtain stability results.
	
	
	For probability laws $\eta,\zeta$ on $\R$ with finite second moments, the quadratic Wasserstein distance is defined by
	\[
	W_2(\eta,\zeta)=\inf_{\pi\in\Pi(\eta,\zeta)}
	\left(\int_{\R^2}|x-y|^2\,\dd\pi(x,y)\right)^{1/2},
	\]
	where $\Pi(\eta,\zeta)$ is the set of probability measures with marginals $\eta$ and $\zeta$. We write $\mathcal L(Y)$ for the law of $Y$, and throughout the probabilistic statements, $X\sim\mu_{n,\nu}$ and $Z\sim T_\nu$. Convex functions take finite real values. Here and throughout, $\stackrel{\mathrm d}{=}$ denotes equality
	in distribution: for real-valued random variables $U$ and $V$,
	this means $\mathbb P(U\le t)=\mathbb P(V\le t)$ for every
	$t\in\mathbb R$.
	
	\begin{theorem}[Optimal mean–variance lower deviations and stability]\label{prob-mean}
		Let $\nu>2$ and let $Y=f(X)$, where $f$ is convex and $0<\operatorname{Var}Y<\infty$. For every $a\in\R$,
		\begin{equation}\label{prob-optimal-mean}
			\mathbb P\!\left(Y\le\mathbb EY-
			\frac{\mathbb E(Z-a)_+}{\sqrt{\operatorname{Var}((Z-a)_+)}}\,\sqrt{\operatorname{Var}Y}\right)
			\le T_\nu(a).
		\end{equation}
		Equality holds if and only if $$Y\stackrel{\mathrm d}=b+c(Z-a)_+\quad \text{for some}\ b\in\R\ \text{and}\ c>0.$$
		The laws are realized by $f(x)=b+c(x_1-a)_+$.
		
		Moreover, for $y<\mathbb EY$ with $\mathbb P(Y\le y)>0$, lut $a=T_\nu^{-1}(\mathbb P(Y\le y))$ and
		\[
		Y_*=y+\frac{\mathbb EY-y}{\mathbb E(Z-a)_+}(Z-a)_+.
		\]
		Then $\mathbb EY_*=\mathbb EY$ and
		\begin{equation}\label{prob-W2}
			W_2^2\bigl(\mathcal L(Y),\mathcal L(Y_*)\bigr)
			\le\operatorname{Var}Y-
			\frac{(\mathbb EY-y)^2\operatorname{Var}((Z-a)_+)}
			{(\mathbb E(Z-a)_+)^2}.
		\end{equation}
	\end{theorem}
	
	\par\smallskip
	We can also approximate the above theorem by the Gaussian case.
	Let $G\sim N(0,1)$. Taking $\nu\to\infty$ in Theorem~\ref{prob-mean} gives the exact Gaussian mean-variance lower-tail curve, see the proof in Corollary \ref{gauss-exact} and a similar form in \eqref{gauss-exact-eq} and \eqref{gauss-W2}. More precisely, let $a_t$ be the unique solution of (the existence and uniqueness is provided by the strictly monotonicity of the function of \(a\) below, see the proof in Corollary \ref{gauss-exact})
	\[
	\frac{\mathbb E(G-a_t)_+}{\sqrt{\operatorname{Var}((G-a_t)_+)}}=t.
	\]
	For $Y=f(G_n)$ with $f$ convex and $0<\operatorname{Var}Y<\infty$, the limiting inequality is
	\begin{equation}\label{intro-gauss-exact}
		\mathbb P(f(G_n)\le\mathbb Ef(G_n)-t\sqrt{\operatorname{Var}f(G_n)})\le\Phi(a_t),\qquad t>0.
	\end{equation}
	Equality holds precisely for the laws $b+c(G-a_t)_+$, $c>0$, so the curve is pointwise optimal. Section~\ref{prob-section} justifies the passage first for convex Lipschitz functions and then by convex approximation; this covers the full Gaussian integrability class. The structure of \eqref{intro-gauss-exact} is consistent with that of \eqref{pv-background}, \eqref{intro-valettas-variance} and \eqref{intro-chen}. The Proposition~\ref{gauss-comparison} shows that \eqref{intro-gauss-exact} strictly improves the explicit estimate \eqref{intro-valettas-variance} when $t>1$ proved by Valettas \cite{Valettas17} (and of course for \eqref{pv-background}), and \eqref{intro-chen} when $t>0$ proved by \cite{Chen26}, that is 
	\begin{equation*}
		\Phi(a_t)<\Phi\!\left(-\frac{t-1}{\sqrt{2\pi}}\right)<\frac{1}{2}e^{-\frac{\pi(t-1)^2}{1024}},
		\qquad t>1,
	\end{equation*}
	and
	\begin{equation*}
		\Phi(a_t)<e^{-t^2/2},
		\qquad t>0.
	\end{equation*}
	\par\smallskip
	
	The restriction $\nu>2$ cannot be removed while retaining a nontrivial variance normalization. Every non-constant finite convex function has a supporting affine function with non-zero slope. Its positive tail therefore dominates a shifted, rescaled Cauchy-type tail, so $\mathbb E|f(X)-b|^q=\infty$ for every $b\in\R$ and $q\ge\nu$. In particular, non-constant convex Cauchy observables have neither a finite absolute first moment nor a finite second moment. A median remains available, and moments of order $0<q<1$ can be finite. This leads to the next theorem, which replaces the mean and variance by the median and a centered moment below the tail parameter. Set $\|U\|_q=(\mathbb E|U|^q)^{1/q}$ for $q>0$; for $q<1$ this notation is only a moment scale.
	
	\begin{theorem}[Optimal centered-moment lower deviations]\label{prob-moment}
		Let $\nu>0$, $0<q<\nu$, and let $Y=f(X)$ for a convex function $f$. Let $m$ be the median of $Y$, and assume $0<\|Y-m\|_q<\infty$. Then, for every $r>0$,
		\begin{equation}\label{prob-optimal-moment}
			\mathbb P\!\left(Y\le m-
			\frac{\|Y-m\|_q}{\|\max(Z,-r)\|_{q}}r\right)
			\le T_\nu(-r).
		\end{equation}
		For each fixed $r>0$, equality holds if and only if
		\[
		Y\stackrel{\mathrm d}=m+c\max(Z,-r)
		\quad\text{for some }\ c>0.
		\]
		These laws are realized by $f(x)=m+c\max(x_1,-r)$.
	\end{theorem}

	
	The moment comparison also reveals the effect of convexity in
	heavy-tailed probability. Under normalization by $\|Y-m\|_q$,
	Markov's inequality gives a bound of order $t^{-q}$.
	The pointwise optimal uniform lower-tail bound in
	\eqref{prob-optimal-moment}, however, has order $t^{-\nu}$,
	recovering the tail exponent of the reference measure.
	In the Cauchy case, parameterize the deviation by
	\[
	t=
	\frac{r}{(\mathbb E|\max\{Z,-r\}|^q)^{1/q}},
	\qquad 0<q<1.
	\]
	The corresponding optimal uniform bound satisfies
	\begin{equation}
		\label{eq:cauchy-optimal-asymptotic}
		T_1(-r)
		\sim
		\frac{\cos(\pi q/2)^{1/q}}{\pi}\,t^{-1}
		\qquad t\to\infty.
	\end{equation}
	A moment of order below one therefore yields a sharp lower-tail
	bound of order $t^{-1}$ with an explicit asymptotic constant.
	For each fixed deviation parameter, letting $\nu\to\infty$
	also gives the Gaussian form of
	\eqref{prob-optimal-moment} and its extremal distributions.

	
	The geometric input behind these sharp probability comparisons is a high-dimensional Ehrhard-type inequality. The classical Ehrhard inequality for Gaussian is the classical source of this
	connection. For standard
	Gaussian measure $\gamma_n$, it states that
	\begin{equation}\label{gaussian}
		\Phi^{-1}\!\left(
		\gamma_n((1-\lambda)A+\lambda B)
		\right)
		\ge
		(1-\lambda)\Phi^{-1}(\gamma_n(A))
		+\lambda\Phi^{-1}(\gamma_n(B)),
		\qquad 0<\lambda<1.
	\end{equation}
	Applied to convex sublevel sets, it implies concavity of
	\[
	s\longmapsto
	\Phi^{-1}\!\left(\mathbb P(f(G_n)\le s)\right)
	\]
	on its effective domain. The associated monotone rearrangement
	represents the distribution of a convex Gaussian observable as a
	non-decreasing convex transform of a one-dimensional Gaussian
	variable. We establish the corresponding mechanism for Cauchy-type measures. Theorem~\ref{prob-representation} then reduces the high-dimensional lower-tail problem to a one-dimensional extremal problem.
	
	Ehrhard proved \eqref{gaussian} for convex sets by symmetrization
	\cite{Ehrhard83}. Lata\l a \cite{Latala96} extended it to the case
	where only one set is convex. Borell subsequently removed the
	convexity assumption by a heat-flow argument \cite{Borell03}. Equality cases were studied by Ehrhard
	\cite{Ehrhard86} and Shenfeld and van Handel
	\cite{ShenfeldVanHandel18}. Beyond strengthening Gaussian
	log-concavity, Ehrhard's inequality connects Gaussian isoperimetry,
	probability estimates for convex functions and boundary
	Poincar\'e-type inequalities \cite{KolesnikovMilman17}.
	Recent work of Livshyts \cite{Livshyts24} investigates stronger
	Gaussian concavity for symmetric convex sets and related equality
	and stability questions.
	
	Outside the Gaussian setting, the absence of a product structure
	creates a substantial obstruction. Byczkowski and \.{Z}ak
	\cite{ByczkowskiZak21} established one-dimensional concavity and
	dilation inequalities for Cauchy-type densities. Byczkowski,
	Ma\l ecki and \.{Z}ak \cite{ByczkowskiMaleckiZak22} subsequently
	formulated the higher-dimensional Cauchy conjecture and proved
	a planar case for rectangles symmetric about a common coordinate
	axis. In a different non-Gaussian setting, Lin and Xing
	\cite{LinXing24} developed symmetrization for compactly supported
	$s$-Gaussian measures, obtaining a planar isoperimetric inequality
	and one-dimensional Ehrhard-type inequalities. Their
	symmetrization does not preserve convexity, which obstructs the
	classical argument in higher dimensions. Fradelizi, Langharst, Liu, Mar\'in Sola and Tang \cite{FLLST} obtained in Lemma 3.2 that the one-dimensional convex measure satisfies the Ehrhard's inequality for convex sets intersected with half-spaces. These developments
	highlight the need to control the dependence of fiber measures on the transverse coordinate.
	
	The following theorem proves the conjectured inequality for Cauchy measure in \cite{ByczkowskiMaleckiZak22}, and extends it to $\mu_{n,\nu}$. This is the first complete sharp Ehrhard-type inequality for the non-product measure in high dimension $n\geq2$,
	together with its equality characterization. It holds in every dimension for arbitrary convex sets and uses the inverse distribution function of the actual one-dimensional marginal. It provides a complete
	convex-set version for a non-Gaussian, non-product family of heavy-tailed measures, including the multivariate Cauchy measure.
	\begin{theorem}[Ehrhard inequality and equality characterization]\label{main}
		Let $n\ge1$, $\nu>0$. For convex sets $A,B\subset\R^n$,
		\begin{equation}\label{ehrhard}
			T_\nu^{-1}\bigl(\mu_{n,\nu}((1-\lambda)A+\lambda B)\bigr)
			\ge(1-\lambda)T_\nu^{-1}(\mu_{n,\nu}(A))+\lambda T^{-1}(\mu_{n,\nu}(B)),\qquad \lambda\in(0,1).
		\end{equation}
		Equality holds if and only if either
		\begin{enumerate}
			\item $\intr(A+B)=\emptyset$, or
			\item $A=\R^n$, or
			\item $\bar{A}=\bar{B}$, or
			\item $\bar{A}$ and $\bar{B}$ are parallel half-spaces with the same normal direction.
		\end{enumerate}
		In particular for convex bodies, equality if and only if $A=B$.
	\end{theorem}
	
	Our proof strategy is to consider a minimization problem with a variational structure. For a fixed unit vector $u$, the main object of our study is the bounded functional
	\[
	K\longmapsto\int_K\arctan\ip{x}{u}\,d\mu_{n,\nu}(x).
	\]
	However, what is ultimately minimized is not itself, but its triple-sum variant over an appropriate class of triples, see in section \ref{sec:variational}. Our proof combines the joint concavity of the one-sided fiber rearrangement with this minimization. The joint concavity of the fiber rearrangement in its endpoints and transverse coordinate ensures preservation of convexity; compactness and an extremal argument then reduce the problem to parallel half-spaces with the same direction $u$.  Finally, we further use the linearity property of parallel half-spaces and a comparison of their quantiles, we derive \eqref{ehrhard}. The equality case is proved using supporting hyperplanes and recession directions. 
	Moreover, as we saw above, the proof method in this paper also applies to the classical Gaussian case as the limiting case $\nu\rightarrow+\infty$, see in Corollary \ref{gaussianlimit}.
	
	The paper is organized as follows. Section \ref{sec:joint} introduces the fiber rearrangement and prove its joint concavity. Section \ref{sec:variational} uses joint concavity to prove Ehrhard's inequality \eqref{ehrhard} in Theorem \ref{main}. Section \ref{sec:equality} establishes the equality characterization for Theorem \ref{main}. Section \ref{prob-section} contains the part on probability: Subsection \ref{sec-mean} proves the optimal mean–variance lower deviation, namely Theorem \ref{prob-mean}; Subsection \ref{sec-moment} proves the optimal centered-moment lower deviations Theorem \ref{prob-moment}; And Section \ref{sec-gaussian-limit} shows that the probability inequalities in the previous two sections can approximate to the Gaussian version as $\nu\rightarrow+\infty$. Finally, Section \ref{sec-comparison} compares our mean–variance lower deviation result with known results.
	
	\section{Fiber rearrangement and its joint concavity}\label{sec:joint}
	
	Byczkowski and \.{Z}ak \cite[Theorem 3.4]{ByczkowskiZak21} proved joint concavity of the fiber-rearrangement for one-dimensional sections of the spherical Cauchy measure. Their Theorems 3.1 and 3.3 establish, respectively, an Ehrhard-type inequality for intervals and a Landau-Shepp-type dilation inequality, the latter being an analogue of the classical Gaussian Landau-Shepp dilation inequality \cite{LS70}. In this section, we extend these results from the Cauchy exponents $(n+1)/2$ to arbitrary $(n+\nu)/2$ with $\nu>0$. Moreover, we establish strict joint concavity for finite interval and a strict dilation inequality, providing the strictness needed for our subsequent equality analysis.
	\begin{figure}[htbp]
		\centering
		\includegraphics[width=0.6\textwidth]{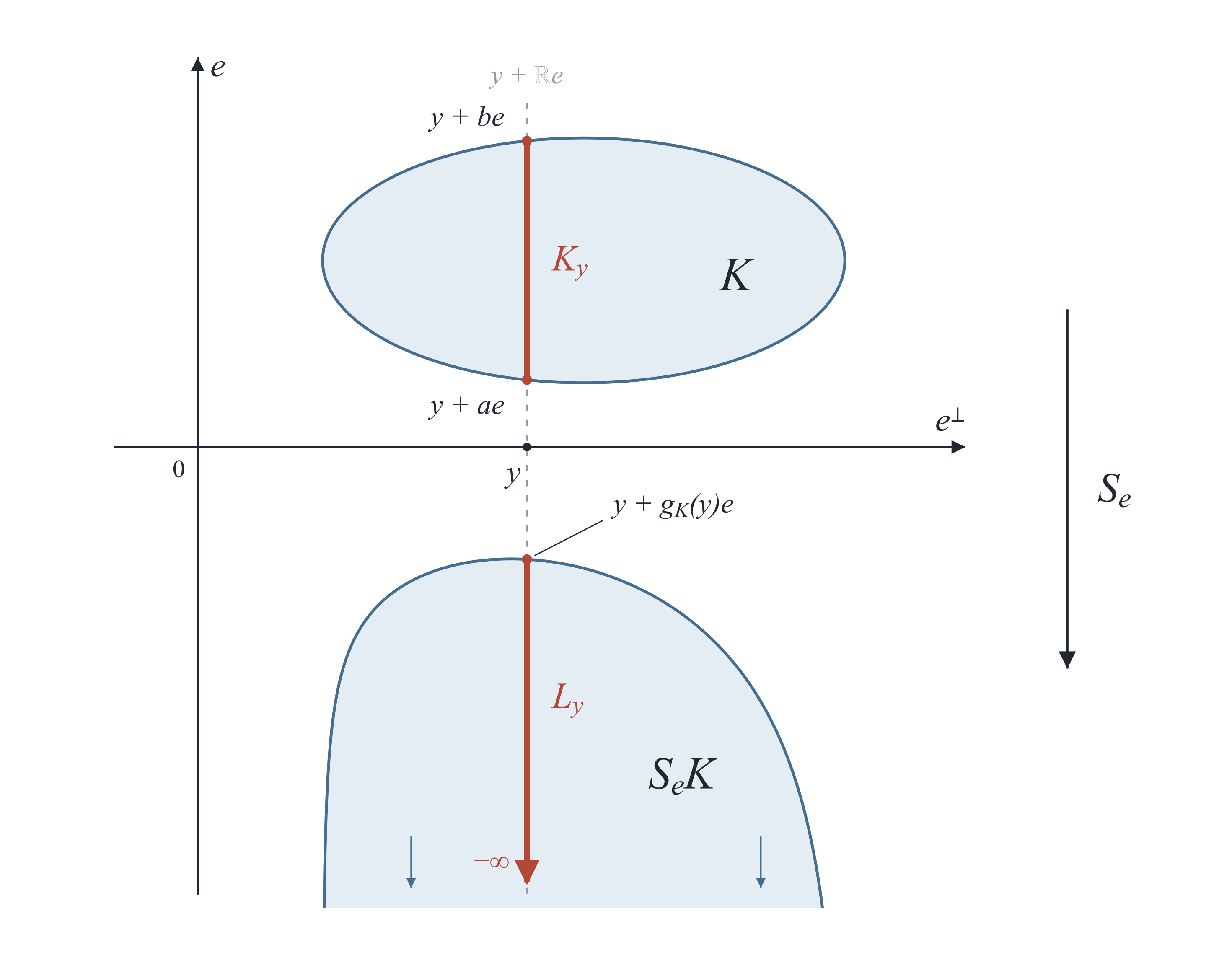}
		\caption{fiber rearrangement}
		\label{fig:fiber rearrangement}
	\end{figure}
	
	Fix a unit vector $e$, and write $x=y+te$ with $y\in e^\perp$. For Cauchy-type measure $\mu_{n,\nu}$ and its density $\rho_{n,\nu}$, set
	\[
	K_y=\{t:y+te\in K\},\qquad
	H_y(t)=\int_{-\infty}^t\rho_{n,\nu}(y+re)\dd r,
	\qquad m_K(y)=\int_{K_y}\rho_{n,\nu}(y+te)\dd t.
	\]
	For a closed convex set $K$, define the \textbf{fiber rearrangement} $S_eK$,
	\begin{equation}\label{symdef}
		S_eK=\{y+te:H_y(t)\le m_K(y)\}.
	\end{equation}
	Detailed illustration is shown in Figure \ref{fig:fiber rearrangement}.
	A zero-mass fiber becomes empty, and a fiber whose mass equals $H_y(+\infty)$ becomes the whole line. Every other fiber becomes $(-\infty,g_K(y)]$, where
	\begin{equation}\label{gKdef}
		H_y(g_K(y))=m_K(y).
	\end{equation}
	For a finite interval $K_y=[a,b]$ of positive length, $g_K(y)=g(y,a,b)$. Take this into \eqref{gKdef}, $g(y,a,b)$ is given by 
	\begin{equation}\label{gdef}
		\int_{-\infty}^{g(y,a,b)}(\nu+|y|^2+t^2)^{-(n+\nu)/2}\dd t
		=\int_a^b(\nu+|y|^2+t^2)^{-(n+\nu)/2}\dd t.
	\end{equation}
	The main result in this section is the following joint fiber concavity.
	\begin{lemma}[Joint fiber concavity]\label{joint}
		Let $\nu>0$. For $y\in\R^{n-1}$ and $a<b$, $g$ is strictly concave on $\{(y,a,b)\in\mathbb{R}^{n+1}:a<b\}$.
	\end{lemma}
	
	By the substitution of variables $t=\sqrt{\nu+|y|^2}s:=c(y)s$, the above integral becomes
	\begin{equation}\label{change of variable}
		\begin{split}
			\int_{-\infty}^{g(y,a,b)}(\nu+|y|^2+t^2)^{-(n+\nu)/2}\dd t&=\int_{-\infty}^{g/c}(c^2+c^2s^2)^{-(n+\nu)/2}cds\\
			&=c^{1-(n+\nu)}\int_{-\infty}^{g/c}(1+s^2)^{-(n+\nu)/2}ds.
		\end{split} 
	\end{equation}
	This suggests that one should first consider concavity of the following form
	\[
	f(t)=(1+t^2)^{-(n+\nu)/2},\qquad
	\qquad L(t)=\int_{-\infty}^t f(r)\dd r.
	\]
	Define
	\begin{equation}\label{qdef}
		q(a,b)=L^{-1}(L(b)-L(a)),\qquad a<b.
	\end{equation}
	\begin{lemma}\label{oned}
		The function $q$ is strictly concave. Moreover,
		\begin{equation}\label{euler}
			a q_a(a,b)+b q_b(a,b)-q(a,b)>0
			\qquad(a<b).
		\end{equation}
		Consequently, $q(ra,rb)>r q(a,b)$ for $r>1$.
	\end{lemma}
	
	\begin{proof}
		Differentiating \eqref{qdef}, we obtain
		\[
		q_a=-\frac{f(a)}{f(q)},\qquad
		q_b=\frac{f(b)}{f(q)},
		\]
		Since $f'(t)=-\frac{(n+\nu)t f(t)}{1+t^2}$, a second differentiation gives
		\[
		\begin{aligned}
			q_{aa}
			&=\frac{(n+\nu)f(a)^2}{f(q)}
			\left(\frac{a}{(1+a^2)f(a)}
			+\frac{q}{(1+q^2)f(q)}\right)\\
			q_{bb}
			&=\frac{(n+\nu)f(b)^2}{f(q)}
			\left(\frac{q}{(1+q^2)f(q)}
			-\frac{b}{(1+b^2)f(b)}\right)\\
			q_{ab}
			&=-\frac{(n+\nu)q f(a)f(b)}{(1+q^2)f(q)^2}.
		\end{aligned}
		\]
		The function 
		\[
		t\longmapsto \frac{t}{(1+t^2)f(t)}
		=t(1+t^2)^{(n+\nu-2)/2}
		\]
		is odd and strictly increasing. Moreover, $L(q)< L(b)$ and
		$L(q)<L(+\infty)-L(a)=L(-a)$, so $q<b$ and $q<-a$.
		It follows that $q_{aa}<0$ and $q_{bb}<0$. To prove that the Hessian is negative definite, it remains to show
		that $q_{aa}q_{bb}-q_{ab}^2>0$. It is equal to prove
		\[
		\frac{q}{(1+q^2)f(q)}< F(a,b), \qquad \text{where}~~F(a,b):=\frac{ab}{a(1+b^2)f(b)-b(1+a^2)f(a)}.
		\]
		Fix $c\in\mathbb R$ and regard $b$ as a function of $a$ on the
		level curve $q=c$. Explicitly,
		\[
		b=b(a)=L^{-1}\bigl(L(c)+L(a)\bigr),\qquad -\infty<a<-c.
		\]
		Here the parameter range follows from
		$L(-c)=L(+\infty)-L(c)$. Along this curve,
		\[
		b'(a)=\frac{f(a)}{f(b)}>0,
		\]
		and direct differentiation yields
		\[
		\frac{d}{da}F(a,b(a))
		=
		\frac{f(a)(a^2-b^2)}
		{\bigl[a(1+b^2)f(b)-b(1+a^2)f(a)\bigr]^2}.
		\]
		Since $a<b$ and $a^2-b^2=(a-b)(a+b)$, the derivative is positive
		when $a+b(a)<0$ and negative when $a+b(a)>0$. Thus $F(a,b(a))$ first increases strictly and then
		decreases strictly, with its unique maximum when $a=-b$.
		
		At the two ends of the parameter interval, $b(a)\to c$ as
		$a\to-\infty$, and $b(a)\to+\infty$ as $a\uparrow-c$.
		Since $tf(t)\to0$ as $|t|\to\infty$, we obtain
		\[
		\lim_{a\to-\infty}F(a,b(a))
		=
		\lim_{a\uparrow-c}F(a,b(a))
		=
		\frac{c}{(1+c^2)f(c)}.
		\]
		Consequently,
		\[
		F(a,b(a))>\frac{c}{(1+c^2)f(c)}
		\qquad(-\infty<a<-c).
		\]
		As $c$ is arbitrary, the determinant is strictly positive throughout
		$\{a<b\}$. Hence the Hessian is negative definite, and $q$ is strictly concave.
		
		For the second assertion, retain the same level curve and compute directly
		\[
		\frac{d}{da}\bigl(bf(b)-af(a)\bigr)
		=
		(n+\nu)f(a)
		\left(\frac{a^2}{1+a^2}-\frac{b^2}{1+b^2}\right).
		\]
		This derivative again has the opposite sign to $a+b$, while 
		\[
		\lim_{a\to-\infty}\bigl(bf(b)-af(a)\bigr)
		=
		\lim_{a\uparrow-c}\bigl(bf(b)-af(a)\bigr)
		=
		cf(c).
		\]
		The same argument therefore gives $bf(b)-af(a)>cf(c)$.
		Since $q=c$ on this curve, the first derivative formulas imply
		\[
		a q_a+b q_b-q
		=
		\frac{bf(b)-af(a)-qf(q)}{f(q)}>0.
		\]
		Finally,
		\[
		\frac{d}{dr}\left(\frac{q(ra,rb)}{r}\right)
		=
		\frac{ra\,q_a(ra,rb)+rb\,q_b(ra,rb)-q(ra,rb)}{r^2}>0
		\qquad(r>0).
		\]
		Comparing with the value at $r=1$ gives
		$q(ra,rb)>r q(a,b)$ for every $r>1$.
	\end{proof}

	Now we turn to Lemma \ref{joint}.
	\begin{proof}[Proof of Lemma~\ref{joint}]
		The change of variable $s=c(y)^{-1}t$ in \eqref{change of variable} gives
		\[
		g(y,a,b)=P(c(y),a,b),\qquad \text{where}\quad
		P(c,a,b)=c\,q\left(\frac{a}{c},\frac{b}{c}\right).
		\]
		The perspective $P$ is concave on $c>0$, $a<b$: for a convex combination of two triples, apply concavity of $q$ with weights proportional to the two values of $s$. Moreover,
		\[
		P_c=q\left(\frac{a}{c},\frac{b}{c}\right)-\frac a c q_a\left(\frac{a}{c},\frac{b}{c}\right)-\frac b c q_b\left(\frac{a}{c},\frac{b}{c}\right)<0
		\]
		by \eqref{euler}. For two triples $(y_i,a_i,b_i)$, let a subscript $\lambda$ denote their weighted average. Convexity of $c$ and the two properties of $P$ give
		\begin{align*}
			g(y_\lambda,a_\lambda,b_\lambda)
			&=P(c(y_\lambda),a_\lambda,b_\lambda)\\
			&\ge P((1-\lambda)c(y_0)+\lambda c(y_1),a_\lambda,b_\lambda)\\
			&\ge(1-\lambda)g(y_0,a_0,b_0)+\lambda g(y_1,a_1,b_1).
		\end{align*}
		If $y_0\ne y_1$, strict convexity of $c$ and $P_c<0$ make the first inequality strict. If $y_0=y_1$, strict concavity of $q$ makes the second inequality strict unless both endpoints also agree. This proves strict joint concavity.
	\end{proof}
	We next shows some properties for fiber rearrangement.
	
	\begin{proposition}\label{symproperties}
		\item (1). The set $S_eK$ is closed, and $\mu(S_eK)=\mu(K)$. 
		\item (2). The rearrangement is monotone under inclusion, that is, if $A\subset B$, then $S_eA\subset S_eB$. 
		\item (3). If $A,B,C$ are closed convex sets and
		\[
		(1-\lambda)A+\lambda B\subset C,
		\]
		then
		\begin{equation}\label{symminkowski}
			(1-\lambda)S_eA+\lambda S_eB\subset S_eC.
		\end{equation}
		\item (4). $S_eK$ is convex.
	\end{proposition}
	
	\begin{proof}
		If $y_j\to y$, closedness gives
		\[
		\limsup_j\mathbf1_K(y_j+te)\le\mathbf1_K(y+te).
		\]
		The densities are bounded by a constant times $(\nu+t^2)^{-(n+\nu)/2}$, which is integrable. The reverse Fatou inequality shows that $m_K$ is upper semi-continuous. The same bound implies continuity of $(y,t)\mapsto H_y(t)$. Thus \eqref{symdef} defines a closed set. Fiberwise preservation of mass and Tonelli's theorem give $\mu(S_eK)=\mu(K)$, including the zero and full fibers. This prove the item (1). And (2) follows directly from the definition.
		
		To prove \eqref{symminkowski} in (3), take $y_i+t_i e\in S_eK_i$, where $K_0=A$ and $K_1=B$. For $\varepsilon>0$,
		\[
		H_{y_i}(t_i-\varepsilon)<m_{K_i}(y_i).
		\]
		Truncation of the original interval fiber therefore supplies finite intervals $[a_i,b_i]\subset(K_i)_{y_i}$ of positive length with
		\[
		g(y_i,a_i,b_i)>t_i-\varepsilon.
		\]
		In this paragraph a subscript $\lambda$ denotes the weighted average with coefficients $1-\lambda$ and $\lambda$. The original inclusion implies the fiber inclusion
		\[
		[a_\lambda,b_\lambda]
		=(1-\lambda)[a_0,b_0]+\lambda[a_1,b_1]
		\subset C_{y_\lambda}.
		\]
		Consequently, Lemma~\ref{joint} gives
		\begin{equation*}
			\begin{split}
				g_C(y_\lambda)&\geq g(y_\lambda,a_\lambda,b_\lambda)\\
				&\geq(1-\lambda)g(y_0,a_0,b_0)+\lambda g(y_1,a_1,b_1)\\
				&>(1-\lambda)t_0+\lambda t_1-\epsilon.
			\end{split}
		\end{equation*}
		Letting $\varepsilon\downarrow0$ proves \eqref{symminkowski}. 
		
		Finally, to prove (4), just Applying (3) with $A=B=C=K$.
	\end{proof}
	
	\section{A variational proof of Theorem \ref{main}}\label{sec:variational}
	In this section, we prove the Ehrhard-type inequality for Cauchy-type measure \eqref{ehrhard} in Theorem \ref{main}. Fix
	$0<\lambda<1$, a unit vector $u\in\mathbb S^{n-1}$, and
	$p=(p_0,p_1,p_2)\in(0,1)^3$. Let
	\begin{equation*}
		\mathcal{T}_p=\bigl\{(K_0,K_1,K_2):\;K_i\subset\mathbb R^n 
		\text{ closed and convex},\ 
		\mu_{n,\nu}(K_i)=p_i\ (i=0,1,2),(1-\lambda)K_0+\lambda K_1\subset K_2\bigr\}.
	\end{equation*}
	Whenever $\mathcal T_p$ is non-empty, we consider the following minimization problem,
	\[
	\inf_{(K_0,K_1,K_2)\in\mathcal T_p}
	\mathcal J_u(K_0,K_1,K_2),
	\]
	where
	\[
	\mathcal J_u(K_0,K_1,K_2)=\sum_{i=0}^2J_u(K_i), \qquad 
	J_u(K)=\int_K\arctan\langle x,u\rangle\,d\mu_{n,\nu}(x).
	\]
	The bounded integrand makes this problem well defined. We first prove that a minimizer exists, then show
	that every minimizing triple consists of parallel half-spaces with the
	same orientation, and finally apply this reduction to Theorem \ref{main}.

	\subsection{Existence of a minimizer}\label{existence}
	We begin with a compactness lemma for closed convex sets of prescribed measure. Recall that, since $\mu$ has a strictly positive density, a convex set has positive $\mu$-measure if and only if its interior is non-empty. We shall also use the facts that every closed convex set $K$ with non-empty interior satisfies
	\begin{equation}\label{convexfacts}
		K=\overline{\intr K}, \qquad \mu_{n,\nu}(\partial K)=0.
	\end{equation}
	\begin{lemma}[Fixed-mass compactness]\label{compactness}
		For $0<p<1$, the family of closed convex sets $K$ with $\mu(K)=p$ is sequentially compact for locally uniform convergence of the distance functions $\dist(\cdot,K)$. Under this convergence,
		\begin{equation}\label{measureconvergence}
			\mu_{n,\nu}(K_j\mathbin\triangle K)\longrightarrow0.
		\end{equation}
		Moreover, the constraint $(1-\lambda)A+\lambda B\subset C$ is closed under this convergence of all three sets.
	\end{lemma}
	
	\begin{proof}
		Let $K_j$ be a sequence of closed convex sets with $\mu(K_j)=p$. We first extract a convergent subsequence. Choose $R>0$ such that $\mu(\R^n\setminus B_R)<p$,
		where $B_R$ is the closed ball with radius $R$ centered at the origin. Every $K_j$ meets $B_R$, hence the distance function 
		\[
		d_j(x)=\operatorname{dist}(x,K_j)
		\]
		are nonnegative, convex and 1-Lipschitz, with $d_j(0)\le R$. By the Arzel\`a–Ascoli theorem and a diagonal argument, we may select to a subsequence, without changing notation, such that $d_j$ converges uniformly on compact sets to a nonnegative convex 1-Lipschitz function $d$.
		
		Let $K=\{d=0\}$. Points of $K_j\cap B_R$ have a convergent subsequence, whose limit lies in $K$. Thus $K$ is nonempty, closed, and convex. The Lipschitz bound gives $d(x)\le\dist(x,K)$. Conversely, for fixed $x$, let $z_j$ be a nearest point of $K_j$ to $x$. Since $|x-z_j|=\dist(x,K_j)$ is bounded, a subsequence converges to $z\in K$, and $|x-z|=d(x)$. Hence $d=\dist(\cdot,K)$. In particular, every $z\in K$ admits approximating points $z_j\in K_j$ with $z_j\to z$.
		
		We next show that the measure is preserved in the limit. If $x\notin K$, then $d(x)>0$, so $x\notin K_j$ for all sufficiently large $j$. Let now $x\in\operatorname{int}K$. Suppose that $x\notin K_j$ along a subsequence.Then there are unit vectors $v_j$ such that
		\[
		\langle v_j,z-x\rangle\le0
		\qquad\text{for every }z\in K_j.
		\]
		Selecting a subsequence, we may assume that $v_j\to v$, where $|v|=1$. Approximate each $z\in K$ by points of $K_j$ and pass to the limit. We obtain
		\[
		\langle v,z-x\rangle\le0
		\qquad\text{for every }z\in K,
		\]
		which contradicts $x\in\operatorname{int}K$. Therefore $x\in K_j$ eventually. It follows that $1_{K_j}\to1_K$ outside $\partial K$. The boundary has $\mu$-measure zero: if $K$ has nonempty interior, this follows from \eqref{convexfacts}, otherwise $K$ is contained in a proper affine subspace. Dominated Convergence Theorem now gives
		\[
		\mu(K_j\triangle K)
		=\int_{\mathbb R^n}|\mathbf1_{K_j}-\mathbf1_K| d\mu_{n,\nu}
		\longrightarrow0.
		\]
		In particular, $\mu_{n,\nu}(K)=p$. This proves the compactness assertion and \eqref{measureconvergence}.
		
		Finally, suppose that $(A_j, B_j, C_j)$ converge to $(A, B, C)$ respectively, and $(1-\lambda)A_j+\lambda B_j\subset C_j$. Take arbitrary $a\in A$ and $b\in B$, choose $a_j\in A_j$ and $b_j\in B_j$ such that $a_j\to a$ and $b_j\to b$. Then 
		\[
		c_j=(1-\lambda)a_j+\lambda b_j\in C_j,
		\qquad
		c_j\longrightarrow c=(1-\lambda)a+\lambda b.
		\]
		Consequently,
		\[
		\operatorname{dist}(c,C)
		=\lim_{j\to\infty}\operatorname{dist}(c,C_j)
		\le\lim_{j\to\infty}|c-c_j|=0.
		\]
		Since $C$ is closed, $c\in C$. Thus
		\[
		(1-\lambda)A+\lambda B\subset C,
		\]
		as required.
	\end{proof}
	Now we can establish the existence of a minimizer.
	Assume $\mathcal T_p\ne\emptyset$ and take a minimizing sequence.
	Applying Lemma~\ref {compactness} successively to its three
	components gives a common subsequence converging to a triple in
	$\mathcal T_p$. Moreover,
	\[
	|J_u(K)-J_u(L)|\le\frac{\pi}{2}\mu(K\triangle L).
	\]
	Thus $\mathcal J_u$ converges along this subsequence.
	Since $|\mathcal J_u|\le3\pi/2$, its infimum is finite and is attained.
	In particular, every non-empty class $\mathcal T_p$ admits a minimizing
	triple.
	
	\subsection{Minimizers consist of parallel half-spaces with the same direction}\label{minimizer structure}
	The characterization of minimizing triples relies on two facts:
	rearrangement strictly decreases $J_u$ unless the set is already
	rearranged, and the common fixed points of the relevant rearrangements
	are halfspaces.
	\begin{lemma}[Directional monotonicity]\label{moment}
		Fix a unit vector $u$, $J_u$ is continuous. If $\ip{e}{u}>0$, then
		\begin{equation}\label{Jineq}
			J_u(S_eK)\le J_u(K).
		\end{equation}
		For a closed convex set of positive measure, equality holds if and only if $S_eK=K$.
	\end{lemma}
	
	\begin{proof}
		Continuity follows from \eqref{measureconvergence} and boundedness of the integrand. We now assume that $\langle e,u\rangle>0$, fix $y\in e^\perp$, and let
		\[
		K_y=\{t\in\mathbb R: y+te\in K\},
		\qquad
		L_y=\{t\in\mathbb R: y+te\in S_eK\}.
		\]
		Write
		\[
		d\nu_y(t)=\rho(y+te) dt,
		\qquad
		\varphi_y(t)=\arctan\bigl(\langle y,u\rangle
		+t\langle e,u\rangle\bigr).
		\]
		Then $\varphi_y$ is strictly increasing, and the definition of $S_eK$ gives
		\[
		\nu_y(K_y)=\nu_y(L_y).
		\]
		If this common mass is zero or equals $\nu_y(\mathbb R)$, the two fibers agree $\nu_y$-almost everywhere, and their integrals are equal. Otherwise, $L_y=(-\infty,h]$ for some $h\in\mathbb R$. Put
		\[
		E=K_y\setminus L_y,
		\qquad
		F=L_y\setminus K_y.
		\]
		By equality of the fiber masses,
		\[
		\nu_y(E)=\nu_y(F).
		\]
		Moreover, $E\subset(h,\infty)$ and $F\subset(-\infty,h]$. Hence
		\begin{align*}
			D(y)
			&:=\int_{K_y}\varphi_y d\nu_y
			-\int_{L_y}\varphi_y d\nu_y\\
			&=\int_E\varphi_y d\nu_y-\int_F\varphi_y d\nu_y\\
			&=\int_E\bigl(\varphi_y(t)-\varphi_y(h)\bigr) d\nu_y(t)
			+\int_F\bigl(\varphi_y(h)-\varphi_y(t)\bigr) d\nu_y(t)\ge0\\
		\end{align*}
		Since $\varphi_y$ is strictly increasing and $\nu_y({h})=0$, equality holds if and only if
		\[
		\nu_y(K_y\triangle L_y)=0.
		\]
		Integrating over $y\in e^\perp$ and applying Fubini's theorem, we obtain
		\[
		J_u(K)-J_u(S_eK)
		=\int_{e^\perp}D(y) dy\ge0.
		\]
		This proves \eqref{Jineq}. Since $D(y)\ge0$, equality holds if and only if the original and rearranged fibers agree $\nu_y$-almost everywhere for almost every $y$, or equivalently,
		\[
		\mu(K\triangle S_eK)=0.
		\]
		
		Finally, suppose that $K$ is closed and convex with positive measure. By Proposition \ref{symproperties}, the same is true for $S_eK$. Two such sets that agree $\mu$-almost everywhere coincide as sets, since $\rho>0$ and each set is the closure of its interior. Thus equality holds if and only if $S_eK=K$.
	\end{proof}
	
	\begin{lemma}[Common fixed points]\label{fixedpoints}
		Let $K$ be a non-empty proper closed convex set. If $S_eK=K$ for every unit vector $e$ with $\ip{e}{u}>0$, then
		\[
		K=\{x:\ip{x}{u}\le a\}
		\]
		for some $a\in\R$.
	\end{lemma}
	
	\begin{proof}
		Every fiber of $S_eK$ is downward closed, so the identity $S_eK=K$ gives $K-te\subset K$ for all $t\ge0$. It follows that 
		$$K+v\subset K,\qquad \ip{v}{u}<0.$$
		If $w\perp u$, apply this to $w-\varepsilon u$ and let $\varepsilon\rightarrow0$. Closedness gives $K+w\subset K$. The same argument with $-w$ gives $K-w\subset K$. Thus,
		\begin{equation*}
			K+w=K,\qquad \forall w\in u^{\perp},
		\end{equation*}which means that $K$ is invariant under every translation in $u^\perp$.
		The intersection of $K$ with the line $\R u$ is non-empty, closed, and downward closed. Translation invariance and properness exclude the whole line. Hence this intersection is $\{tu:t\le a\}$ for some finite $a$, which finish the proof.
	\end{proof}
	
	Fix a minimizing triple $(K_0,K_1,K_2)\in\mathcal T_p$.
	For any unit vector $e$ with $\langle e,u\rangle>0$,
	Proposition \ref{symproperties} shows that
	$(S_eK_0,S_eK_1,S_eK_2)$ still belongs to $\mathcal T_p$.
	Minimality and Lemma~\ref{moment} give
	\[
	\mathcal J_u(K_0,K_1,K_2)
	\le \mathcal J_u(S_eK_0,S_eK_1,S_eK_2)
	\le \mathcal J_u(K_0,K_1,K_2).
	\]
	Each term in the sum is non-increasing, so equality of the sums forces
	$J_u(S_eK_i)=J_u(K_i)$ for every $i=0,1,2$.
	Consequently,
	$$S_eK_i=K_i,\qquad i=0,1,2,\quad \forall \ip{e}{u}>0.$$
	Since $0<p_i<1$, each component is nonempty
	and proper. Lemma~\ref{fixedpoints} yields
	\[
	K_i=H_{a_i}^u,\qquad i=0,1,2,
	\]
	for finite $a_i$. This proves that every minimizing triple
	consists of parallel halfspaces with the common orientation $u$.

	\subsection{Proof of \eqref{ehrhard} in Theorem \ref{main}}
	Suppose first that $A,B$ are closed and $0<\mu_{n,\nu}(A),\mu_{n,\nu}(B)<1$. Let $C=\overline{(1-\lambda)A+\lambda B}$, and we can assume $0<\mu_{n,\nu}(C)<1$. Fix $u\in S^{n-1}$, $p=(p_0,p_1,p_2)=\left(\mu_{n,\nu}(A),\mu_{n,\nu}(B),\mu_{n,\nu}(C)\right)$. Taking closures preserves the measures of convex sets, and
	\[
	(1-\lambda)\overline A+\lambda\overline B\subset C.
	\]
	Hence \((\overline A,\overline B,C)\in\mathcal T_p\) and shows that $\mathcal{T}_p$ is non-empty. Combining subsections
	\ref{existence} and \ref{minimizer structure},
	it provides a minimizer
	$(H_{a_0}^u,H_{a_1}^u,H_{a_2}^u)$.
	
	We record the marginal identity that converts half-space
	thresholds into quantiles. The beta integral
	\[
	\int_{\mathbb R^d}(c+|y|^2)^{-\beta}\,dy
	=\pi^{d/2}\frac{\Gamma(\beta-d/2)}{\Gamma(\beta)}
	c^{d/2-\beta},
	\qquad \beta>d/2,\quad c>0,
	\]
	with $d=n-1$, $c=\nu+t^2$, and $\beta=(n+\nu)/2$ gives
	\[
	\int_{u^\perp}\rho(tu+y)\,dy=\frac{\Gamma((\nu+1)/2)}{\sqrt{\nu\pi}\Gamma(\nu/2)}\left(1+\frac{t^2}{\nu}\right)^{-(\nu+1)/2}=T'_\nu(t).
	\]
	For $n=1$, this is understood as a zero-dimensional integral. Thus
	\[
	\mu(H_a^u)=T_\nu(a),\qquad a\in\mathbb R,
	\]
	which means that $a_i=T_{\nu}^{-1}(p_i)$. 
	
	Furthermore, the linearity property of half-spaces with the same direction gives
	\begin{equation}\label{linear of halfspace}
		(1-\lambda)H_{a_0}^u+\lambda H_{a_1}^u
		=H_{(1-\lambda)a_0+\lambda a_1}^u.
	\end{equation}
	And the inclusion shows that 
	\begin{equation}\label{inclusion for halfspace}
		(1-\lambda)H_{a_0}^u+\lambda H_{a_1}^u\subset H^u_{a_2}.
	\end{equation}
	Combining \eqref{linear of halfspace}, \eqref{inclusion for halfspace} and $a_i=T_{\nu}^{-1}(p_i)$, it deduces that
	\[
	(1-\lambda)T_{\nu}^{-1}(p_0)+\lambda T_{\nu}^{-1}(p_1)
	\le T_{\nu}^{-1}(p_2),
	\]
	which is precisely \eqref{ehrhard}.
	
	It remains to consider the endpoint cases. If \(\mu_{n,\nu}(A)=0\) or \(\mu_{n,\nu}(B)=0\), the right-hand side of the inequality is defined to be \(-\infty\), and the assertion is immediate. Now suppose that both measures are positive and, say \(\mu_{n,\nu}(A)=1\), then \(A=\mathbb R^n\). Since \(B\ne\varnothing\), we obtain
	\[
	(1-\lambda)\mathbb R^n+\lambda B=\mathbb R^n.
	\]
	Thus both sides of the desired inequality are \(+\infty\). The case \(\mu(B)=1\) is identical. \qed

	\begin{corollary}\label{gaussianlimit}
		Theorem~\ref{main} implies the Gaussian Ehrhard inequality \eqref{gaussian} for convex sets.
	\end{corollary}
	
	\begin{proof}
		As $\nu\to\infty$, the density in \eqref{student} converges pointwise to the standard Gaussian density, by the gamma-ratio asymptotics and $(1+r/\nu)^{-\nu/2}\to e^{-r/2}$. Since the densities all integrate to one, Scheff\'e's lemma \cite{S47} gives convergence in total variation. Thus $\mu_{n,\nu}(K)\to\gamma_n(K)$ for every measurable $K$. Also $T_\nu\to\Phi$ pointwise, and strict monotonicity of $\Phi$ implies
		\[
		T_\nu^{-1}(p_\nu)\longrightarrow\Phi^{-1}(p)
		\quad\text{whenever}\ p_\nu\to p\in(0,1).
		\]
		This follows by bracketing $\Phi^{-1}(p)$ between two nearby points and using convergence of the distribution functions there. Apply these facts to \eqref{ehrhard} when all three Gaussian measures are in $(0,1)$. A sum of full Gaussian measure makes the desired inequality immediate; zero or full input measures are covered by the endpoint convention and the convex-set argument above.
	\end{proof}
	
	\section{Equality characterization}\label{sec:equality}
	This section gives the equality characterization.
	For a non-empty closed convex set $K$, we introduce the support function and recession cone
	\[
	h_K(v)=\sup_{x\in K}\ip{v}{x},\qquad
	\rec K=\{d:K+td\subset K\text{ for every }t\ge0\}.
	\]
	The recession cone $\rec K$ is a closed convex cone. Fibers are taken in direction $e$, with the notation of Section~\ref{sec:joint}.
	
	Before considering the equality characterization of Ehrhard’s inequality, we need to give a gap estimate between the rearranged set and the original set.
	\begin{lemma}[Gap estimate]\label{gap}
		Suppose $h_K(v)<\infty$ and $\ip{v}{e}>0$.
		\item (1). Every nonempty fiber of $K$ has a finite upper endpoint, bounded by
		\begin{equation}\label{upper}
			L_K(y)=\frac{h_K(v)-\ip{v}{y}}{\ip{v}{e}};
		\end{equation}
		\item (2). If $-e\in \rec K$, every non-empty fiber is the lower half-line;
		\item (3). If $-e\notin\rec K$, every nonempty fiber is bounded, and there is $\delta>0$ such that every positive-mass fiber satisfies
		\begin{equation}\label{uniformgap}
			g_K(y)\le L_K(y)-\delta\sqrt{\nu+|y|^2}.
		\end{equation}
	\end{lemma}
	
	\begin{proof}
		For any $x\in K$, let $x=y+te$. The definition of support function gives 
		\begin{equation*}\label{gap support inequality}
			\ip{v}{y+te}\le h_K(v),
		\end{equation*} which is directly \eqref{upper}. And by the definition of the recession cone, (2) is true.
		
		We now prove (3). If $-e\notin\rec K$, choose $x\in K$ and $r>0$ with $z=x-re\notin K$. Let $p$ be the nearest point of $K$ to $z$, which means $h_K(z-p)=\ip{z-p}{p}$. Denote $w=z-p$, then
		\[
		\ip{w}{z}=\ip{z-p}{z}=\ip{z-p}{p}+\ip{z-p}{z-p}=h_K(w)+|z-p|^2>h_K(w).
		\]
		Since $x\in K$ and $z=x-re$, we have $\ip{w}{e}<0$. Using the same thing at \eqref{gap support inequality} by replacing $v$ to $w$, gives a lower bound on every fiber:
		\begin{equation*}
			\begin{split}
				t&\ge\frac{h_K(w)-\ip{w}{y}}{\ip{w}{e}}\\
				&\geq\frac{|w||y|+|h_K(w)|}{\ip{w}{e}}\\
				&\geq\frac{\left(|w|+\frac{|h_K(w)|}{\sqrt{v}}\right)}{\ip{w}{e}}\sqrt{v+|y|^2}
				=-M\sqrt{\nu+|y|^2}
			\end{split}
		\end{equation*}
		for some $M>0$ not depends on $y$. Together with \eqref{upper}, this shows that every nonempty fiber is bounded from above and below.
		
		Consider a positive mass fiber $[a,b]$. By the definition of $g_K$, the mass of $[g,b]$ equals to $(-\infty,a]$. By a rescaling, we deduce that
		\[
		\int_{\frac{g_K(y)}{\sqrt{\nu+|y|^2}}}^{\frac{b}{\sqrt{\nu+|y|^2}}}(1+t^2)^{-\beta}\dd t
		=\int_{-\infty}^{\frac{a}{\sqrt{\nu+|y|^2}}}(1+t^2)^{-\beta}\dd t
		\ge\int_{-\infty}^{-M}(1+t^2)^{-\beta}\dd t=:\delta>0.
		\]
		The integrand on the left is at most one, so $\frac{b-g_K(y)}{\sqrt{v+|y|^2}}\ge\delta$. Since $b\le L_K(y)$, this proves \eqref{uniformgap}.
	\end{proof}
	We next show the closeness when at least one of fiber of the set is closed interval.
	\begin{lemma}[Closedness of the rearranged sum]\label{closedsum}
		Suppose $h_A(v),h_B(v)<\infty$ and $\ip{v}{e}>0$. If $-e$ does not belong to at least one of $\rec A$ and $\rec B$, then
		\[
		(1-\lambda)S_eA+\lambda S_eB
		\]
		is closed.
	\end{lemma}
	
	\begin{proof}
		Let $-e\notin\rec B$, and use the constant $\delta$ from Lemma~\ref{gap} for $B$. Let 
		\[
		x_j=(1-\lambda)\left((y_{0})_j+(t_{0})_je\right)
		+\lambda\left((y_{1})_j+(t_{1})_je\right)\longrightarrow x,
		\]
		where $x_j\in(1-\lambda)S_eA+\lambda S_eB$, $(y_0)_j+(t_0)_je\in S_eA$ and $(y_1)_j+(t_1)_je\in S_eB$. We next prove that both limits of the two sequence still in $S_eA$ and $S_eB$. Write $x_j=y_j+t_je$. Since $x_j\rightarrow x$, thus $y_j$ and $t_j$ are bounded. Since fibers of $A$ and $B$ have positive mass, Lemma~\ref{gap} implies
		\[
		(t_{0})_j\le L_A((y_{0})_j),\qquad
		(t_{1})_j\le L_B((y_{1})_j)-\delta\sqrt{\nu+|(y_{1})_j|^2}.
		\]
		The two affine upper bounds have the same linear part of $y$. Therefore
		\begin{equation}\label{boundeddecomposition}
			\lambda\delta\sqrt{\nu+|(y_{1})_j|^2}
			\le\frac{(1-\lambda)h_A(v)+\lambda h_B(v)-\ip{v}{y_j}}
			{\ip{v}{e}}-t_j.
		\end{equation}
		The right-hand side is bounded, so $y_{1j}$ is bounded. The identity $y_j=(1-\lambda)(y_{0})_j+\lambda (y_{1})_j$ also bounds $(y_{0})_j$. The affine bounds now bound both $(t_{i})_j$ from above. Since their weighted average $t_j$ converges, both are bounded from below as well. After passage to a subsequence, the two summands converge in the closed sets $S_eA$ and $S_eB$. Thus $x\in (1-\lambda)S_e A+\lambda S_eB$.
	\end{proof}
	
	\begin{proof}[Proof of the equality characterization in Theorem~\ref{main}]
		
		Without loss of generality, we may assume that $A,B$ are closed. Suppose first that $0<\mu(A),\mu(B)<1$. Let equality holds in \eqref{ehrhard}. Denote $C=\overline{(1-\lambda)A+\lambda B}$. Equality implies $0<\mu(C)<1$.  Since $(1-\lambda)A+\lambda B\subset C$, for any unit vector $e$, we can deduce from Proposition~\ref{symproperties} that
		\begin{equation}\label{include ienquality in equality statement}
			\overline{(1-\lambda)S_eA+\lambda S_eB}\subset S_eC.
		\end{equation}
		On the other side, we apply the Ehrhard's inequality \eqref{ehrhard} on the sets $S_eA, S_eB$:
		\begin{equation*}
			\begin{split}
				T\circ\mu(({1-\lambda})S_eA+\lambda S_eB)&\geq (1-\lambda)T\circ\mu(S_eA)+ \lambda T\circ\mu(S_eB)\\
				&=(1-\lambda)T\circ\mu(A)+ \lambda T\circ\mu(B)\\
				&=T\circ\mu(C)=T\circ\mu(S_eC),
			\end{split}
		\end{equation*}
		which means that the measure of $(1-\lambda) S_eA+\lambda S_eB$ is bigger or equal to $S_eC$. Combining with \eqref{include ienquality in equality statement}, we have
		\begin{equation}\label{equalityclosure}
			\overline{(1-\lambda)S_eA+\lambda S_eB}=S_eC.
		\end{equation}
		
		Since $C$ is closed and $\mu(C)<1$, there exists a unit vector $v$ with $h_C(v)<\infty$. By the definition of support function,
		\[
		h_C(v)=(1-\lambda)h_A(v)+\lambda h_B(v).
		\]
		This shows that $h_A(v),h_B(v)<+\infty$. We next  claim that:
		
		\textbf{Claim}: Fix any unit vector $e$ with $\ip{v}{e}>0$. Then
		\begin{equation}\label{dichotomy}
			A=B\quad\text{or}\quad -e\in\rec A\cap\rec B.
		\end{equation}
		
		We prove the claim. Suppose first that $-e\in\rec A$ but $-e\notin\rec B$. From Lemma~\ref{closedsum}, \eqref{equalityclosure} becomes
		\begin{equation}\label{no bar equal}
			(1-\lambda)S_eA+\lambda S_eB=S_eC.
		\end{equation} We choose any positive-mass fiber of $C$, and let $x=y+te$. By \eqref{upper} in Lemma~\ref{gap}, the common supporting bound $L_C(y)$ makes its rearranged upper endpoint $g_C(y)$ finite. By \eqref{no bar equal}, The boundary point has an actual representation
		\begin{equation}\label{actualrep}
			\begin{split}
				y+g_C(y)e=(1-\lambda)(y_0+t_0e)+\lambda(y_1+t_1e),\\
				y_0+t_0e\in S_eA,\qquad y_1+t_1e\in S_eB.
			\end{split}
		\end{equation}
		Since $-e\in\rec A\setminus\rec B$, by the definition of recession cone, the fibers have the following form
		\[
		A_{y_0}=(-\infty,b_0],\qquad B_{y_1}=[a_1,b_1],\qquad a_1<b_1.
		\]
		Note that $g_B(y_1)<b_1$, because the symmetrization transfers part of its mass to negative part. Hence \eqref{actualrep} gives
		\[
		g_C(y)=(1-\lambda)t_0+\lambda t_1\leq(1-\lambda)b_0+\lambda g_B(y_1)
		<(1-\lambda)b_0+\lambda b_1.
		\]
		On the other hand, the Minkowski inclusion $(1-\lambda)A+\lambda B\subset C$ gives
		\[
		(-\infty,(1-\lambda)b_0+\lambda b_1]\subset C_y,
		\]
		which is a contradiction.
		
		Suppose next that $-e$ belongs to neither recession cone, and we need to prove $A=B$. Again \eqref{actualrep} holds for every positive-mass fiber of $C$, and both fibers of $A$ and $B$ are finite intervals $[a_0,b_0]$ and $[a_1,b_1]$ with $a_i<b_i$. Let $a_\lambda=(1-\lambda)a_0+\lambda a_1$ and $b_\lambda=(1-\lambda)b_0+\lambda b_1$. Thus,
		\begin{align*}
			g_C(y)&\ge g(y,a_\lambda,b_\lambda)\\
			&\ge(1-\lambda)g(y_0,a_0,b_0)+\lambda g(y_1,a_1,b_1)\\
			&\ge(1-\lambda)t_0+\lambda t_1=g_C(y),
		\end{align*}
		here the second inequality is by the joint concavity in Lemma~\ref{joint}. So all inequalities are equalities. Strict joint concavity again in Lemma \ref{joint} implies
		\[
		(y_0,a_0,b_0)=(y_1,a_1,b_1).
		\]
		In particular $y_0=y_1=y$, and $A_y=B_y=[a_0,b_0]$. This interval is contained in $C_y$ and has the same fiber measure as $C_y$, since its rearranged endpoint equals $g_C(y)$. Convexity and positivity of the density imply $$C_y=A_y=B_y,$$
		which gives $A=B=C$. This finish the proof of \textbf{Claim}.
		
		If $A\ne B$, from \textbf{Claim}, we have $e\in\rec A\cap\rec B$ for every $\ip{v}{e}>0$. Since recession cones are closed cones, we deduce that
		\[
		\{d:\ip{v}{d}\le0\}\subset\rec A\cap\rec B.
		\]
		In particular, both sets are invariant under all translations in $v^\perp$, and $K-tv\subset K$ for $K=A,B$ and $t\ge0$. Since $h_A(v),h_B(v)<+\infty$, this means
		\[
		A=\{x:\ip{v}{x}\le h_A(v)\},\qquad
		B=\{x:\ip{v}{x}\le h_B(v)\}.
		\]
		We complete the necessity for closed set and $\mu(A),\mu(B)\in(0,1)$, and the sufficiency can be obtained directly by computation. 
		
		For arbitrary convex sets whose measures lie in $[0,1]$, it is easy to see that, if $\mu(A)=0$, then  $\intr(A+B)=\emptyset$; if $\mu(A)=1$, then $A=\mathbb{R}^n$. And when $A,B$ are convex bodies, the case (1), (2) and (4) will not happen, then $A=B$. We finish the whole proof.
	\end{proof}

	\Needspace{8\baselineskip}
	\section{lower deviations on convex observables}
	\label{prob-section}
	
	In this section, we set $X\sim\mu_{n,\nu}$, $Z$ have distribution function $T_\nu$, and $f:\mathbb{R}^n\rightarrow\mathbb{R}$ is convex. We first show that high dimensional convex observable $f(X)$ has the same distribution with $g(Z)$ for some convex, non-decreasing \(g\).   
	
	\begin{theorem}[Distributions of convex observables]\label{prob-representation}
		For every convex $f:\R^n\to\R$, there is a non-decreasing convex $g:\R\to\R$ such that
		\begin{equation}\label{prob-representation-eq}
			f(X)\stackrel{\mathrm d}=g(Z).
		\end{equation}
		Conversely, every such distribution is realized by $f(x)=g(x_1)$.
	\end{theorem}
	
	\begin{proof}
		The assertion is immediate when $f$ is constant. Otherwise, put $Y=f(X)$, $F(t)=\mathbb P(Y\le t)$ and $\ell=\inf f$. A finite nonconstant convex function on $\R^n$ is continuous and unbounded above. For $t>\ell$, the convex set $K_t=\{f\le t\}$ has nonempty interior, and $\{f=t\}\subset\partial K_t$. Indeed, choose $x_0$ with $f(x_0)<t$. If $f(x)=t$ at an interior point of $K_t$, the segment from $x_0$ through $x$ can be extended slightly within $K_t$, contradicting convexity. Since the boundary of a convex set with nonempty interior is Lebesgue null, $F$ is continuous on $(\ell,\infty)$. If $\ell<s<t$, continuity of $f$ and its interval-valued range give an open set on which $s<f<t$. Since the density is positive, so $F(s)<F(t)$.
		
		The inclusion
		\[
		(1-\lambda)K_s+\lambda K_t\subset K_{(1-\lambda)s+\lambda t}
		\]
		and Ehrhard's inequality in Theorem~\ref{main} show that 
		$$t\rightarrow T_\nu^{-1}\circ F(t)$$
		is concave on $(\ell,\infty)$. If $\ell$ is finite, let $\alpha=\mathbb P(Y=\ell)$; otherwise put $\alpha=0$. The continuous strictly increasing function $T^{-1}_\nu\circ F$ has range $(T_\nu^{-1}(\alpha),\infty)$, and its inverse is increasing and convex. When $\alpha>0$, extend this inverse by the constant $\ell$ to the left of $T_\nu^{-1}(\alpha)$. Its right limit at that endpoint is $\ell$ and its slopes are nonnegative, so the extension remains convex. When $\alpha=0$, the inverse is already defined on $\R$.
		Choose $g$
		\[
		g:=F^{-1}\circ T_\nu,\qquad\text{where}\quad F^{-1}(u)=\inf\{t:F(t)\ge u\},\quad 0<u<1.
		\]
		Thus $g$ is finite, non-decreasing and convex. And
		\begin{equation*}
			\begin{split}
				\mathbb{P}(g(Z)\leq t)&=\mathbb{P}(F^{-1}(T_\nu(Z))\leq t)\\
				&=\mathbb{P}(T_\nu(Z)\leq F(t))\\
				&=\mathbb{P}(Z\leq T^{-1}_\nu\circ F(t))=T_\nu(T_\nu^{-1}(F(t)))\\
				&=F(t)=\mathbb{P}(f(X)\leq t),
			\end{split}
		\end{equation*}gives \eqref{prob-representation-eq}. Conversely, the first marginal of $X$ has distribution $T_\nu$, so $f(x)=g(x_1)$ realizes every such law.
	\end{proof}
	
	\begin{remark}\label{g(a)=y}
		For later use, if $g$ is non-constant and $0<\mathbb P(g(Z)\le y)<1$, the set $\{z:g(z)\le y\}$ is a closed half-line $(-\infty,a]$ and $g(a)=y$. Indeed, monotonicity makes it a lower interval, continuity closes it, and the two probability conditions make its endpoint finite. Continuity and the definition of the endpoint then give $g(a)=y$. Consequently,
		\begin{equation}\label{prob-threshold}
			a=T_\nu^{-1}(\mathbb P(g(Z)\le y)),\qquad g(a)=y.
		\end{equation}
	\end{remark}
	
	\subsection{The mean--variance bound}\label{sec-mean}
	\begin{proof}[Proof of Theorem~\ref{prob-mean}]
		We first proof the stability. By Theorem~\ref{prob-representation}, there is a non-decreasing convex function $g:\mathbb{R}\rightarrow\mathbb{R}$ such that $f(X)\overset{\mathrm d}=g(Z)$, where $Z$ follows the one-dimensional Cauchy-type distribution. By the assumption $\mathbb{P}(Y\leq y)>0$ and Remark \ref{g(a)=y}, $g(a)=y$. By the definition of Wasserstein distance $W_2$, 
		\begin{equation}\label{Wasserstein distance}
			\begin{split}
				W^2_2(\mathcal{L}(Y),&\mathcal{L}(Y^*))\leq \mathbb E|Y-Y^*|^2\\
				&=\mathbb E\left|Y-y-\frac{\mathbb EY-y}{\mathbb E(Z-a)_+}(Z-a)_+\right|^2\\
				&=\mathbb E(Y-y)^2-2\mathbb E\left(\frac{\mathbb EY-y}{\mathbb E(Z-a)_+}(Y-y)(Z-a)_+\right)+\frac{(\mathbb EY-y)^2}{(\mathbb E(Z-a)_+)^2}\mathbb E(Z-a)_+^2\\
				&=\operatorname{Var}(Y)+(\mathbb EY-y)^2-2\frac{\mathbb EY-y}{\mathbb E(Z-a)_+}\mathbb E\left((Y-y)(Z-a)_+\right)+\frac{(\mathbb EY-y)^2}{(\mathbb E(Z-a)_+)^2}\mathbb E(Z-a)_+^2.\\
			\end{split}
		\end{equation}
		From the above equation, we should estimate $\mathbb E((Y-y)(Z-a)_+)$. More precisely, we should derive a lower bound for it. 
		
		Consider the probability measure $d\mu(z)=\frac{z-a}{\mathbb E(z-a)_+}dT_\nu(z)$, and its domain is $(a,\infty)$. Since $g$ is non-decreasing convex function, the function 
		\begin{equation*}
			\frac{Y-y}{z-a}=\frac{g(z)-g(a)}{z-a}
		\end{equation*}
		is non-decreasing and non-negative on $(a,+\infty)$. Notice that $z-a$ also has this property. By Chebyshev covariance formula, the covariance between these two Random variable on $\mu$ should bigger than $0$:
		\begin{equation*}
			\operatorname{Cov}_\mu\left(\frac{Y-y}{(z-a)_+},z-a\right)\geq0,
		\end{equation*}
		which is equivalent to
		\begin{equation}\label{covariance formula}
			\mathbb E_\mu\left(Y-y\right)\geq \mathbb E_\mu\left(\frac{Y-y}{(z-a)_+}\right)\mathbb E_\mu(z-a).
		\end{equation}
		We need to prove that these expectations are all finite, and we use the Tail-integral formula:
		\begin{equation*}
			\begin{split}
				\mathbb E_\mu(Y-y)&=\int_{a}^{+\infty}(Y-y)\frac{z-a}{\mathbb E(z-a)_+}dT_\nu(z)=\frac{\mathbb E((Y-y)(z-a)_+)}{\mathbb E(z-a)_+}<+\infty,\\
				\mathbb E_\mu\left(\frac{Y-y}{(z-a)_+}\right)&=\int_a^{+\infty}\frac{Y-y}{(z-a)_+}\cdot\frac{z-a}{\mathbb E(z-a)_+}dT_\nu(z)=\frac{\mathbb E(Y-y)_+}{\mathbb E(z-a)_+}<+\infty,\\
				\mathbb E_\mu(z-a)&=\int_a^{+\infty}(z-a)\frac{z-a}{\mathbb E(z-a)_+}dT_\nu(z)=\frac{\mathbb E(z-a)_+^2}{\mathbb E(z-a)_+}<+\infty.
			\end{split}
		\end{equation*}
		The finiteness of the last two expectations follows directly from $\nu>2$ and $\operatorname{Var}Y\in(0,\infty)$, and for the first one, we use the Cauchy-Schwarz inequality:
		\begin{equation*}
			\mathbb E((Y-y)(z-a)_+)^2\leq \mathbb E(Y-y)^2\mathbb E(z-a)_+^2<+\infty.
		\end{equation*}
		Taking these expression into \eqref{covariance formula}, we deduce that
		\begin{equation}
			\label{covariance inequality after tail}
			\mathbb E((Y-y)(z-a)_+)\geq \frac{\mathbb E(Y-y)_+\mathbb E(z-a)_+^2}{\mathbb E(z-a)_+}=\frac{\mathbb EY-y}{\mathbb E(z-a)_+}\mathbb E(z-a)_+^2.
		\end{equation}
		Finally, we bring \eqref{covariance inequality after tail} into \eqref{Wasserstein distance} and deduce that
		\begin{equation*}
			\begin{split}
				W^2_2(\mathcal{L}(Y),\mathcal{L}(Y^*))&\leq \operatorname{Var}Y+(\mathbb EY-y)^2\\
				&-2\frac{\mathbb EY-y}{\mathbb E(Z-a)_+}\cdot\frac{\mathbb EY-y}{\mathbb E(Z-a)_+}\mathbb E(Z-a)^2_++\frac{(\mathbb EY-y)^2}{(\mathbb E(Z-a)_+)^2}\mathbb E(Z-a)_+^2\\
				&=\operatorname{Var}Y+(\mathbb EY-y)^2-\frac{(\mathbb EY-y)^2}{(\mathbb E(Z-a)_+)^2}\mathbb E(Z-a)_+^2\\
				&=\operatorname{Var}Y+\frac{(\mathbb EY-y)^2\operatorname{Var}(Z-a)_+}{(\mathbb E(Z-a)_+)^2},
			\end{split}
		\end{equation*}
		which is exactly the stability \eqref{prob-W2}.
		
		We next prove the lower deviations inequality \eqref{prob-optimal-mean} by using stability \eqref{prob-W2}. Consider the coefficient function
		\begin{equation*}
			R(a):=\frac{\mathbb E(Z-a)_+}{\sqrt{\operatorname{Var}(Z-a)_+}}, \quad a\in\mathbb{R}.
		\end{equation*}
		Set $y=\mathbb EY-R(a)\sqrt{\operatorname{Var}Y}$, and \eqref{prob-optimal-mean} is equivalent to 
		\begin{equation}\label{equivalent form of 1.11}
			\mathbb{P}(Y\leq y)\leq T_\nu(a).
		\end{equation}
		One can see that this $y$ satisfies that $y<\mathbb EY$ because of $R(a),\operatorname{Var}Y>0$. Without loss of generality, we can assume that $\mathbb{P}(Y\leq y)>0$. Due to the duplication of notation, we redefine $b=T_\nu^{-1}(\mathbb{P}(Y\leq y))$. Taking $y$ and $b$ into \eqref{prob-W2}, we deduce that
		\begin{equation*}
			0\leq W^2_2(\mathcal{L}(Y),\mathcal{L}(Y^*))\leq \operatorname{Var}Y-\frac{(\mathbb EY-y)^2}{R(b)^2}=\operatorname{Var}Y-\operatorname{Var}Y\left(\frac{R(a)}{R(b)}\right)^2.
		\end{equation*}
		Since $\operatorname{Var}Y>0$, we have
		\begin{equation*}
			R(b)\geq R(a).
		\end{equation*}
		\textbf{Claim}: $R$ is a strictly decreasing function, and its range is $(0,+\infty)$. 
		
		If the claim is true, then $b\leq a$. Combining the fact that $T_\nu$ is strictly increasing, we have
		\begin{equation*}
			P(Y\leq y)=T_\nu(b)\leq T_\nu(a),
		\end{equation*}
		which is \eqref{equivalent form of 1.11}. And for the equality characterization, the necessity can be deduced from the strict monotonicity of $R$ and $T_\nu$ that $a=b$, which means that in \eqref{prob-W2} we have
		\begin{equation*}
			0\leq \mathbb E\left|Y-y-\frac{\mathbb EY-y}{\mathbb E(Z-a)_+}(Z-a)_+\right|^2\leq0.
		\end{equation*}
		That is to say, $Y=g(Z)=y+c(Z-a)_+$ almost surely. On the other side, if $Y=g(Z)=y+c(Z-a)_+$ almost surely for some $y\in\mathbb{R}$ and $c>0$, then 
		\begin{equation*}
			\begin{split}
				EY&=y+cE(Z-a)_+,\\
				\operatorname{Var}Y&=c^2\operatorname{Var}(Z-a)_+.
			\end{split}
		\end{equation*}
		Then
		\begin{equation*}
			\mathbb{P}\left(Y\leq EY-\frac{E(Z-a)_+}{\sqrt{\operatorname{Var}(Z-a)_+}}\sqrt{\operatorname{Var}Y}\right)=\mathbb{P}(Y\leq y)=\mathbb{P}(Z\leq a)=T_\nu(a).
		\end{equation*}
		This finish the equality characterization.
		
		Finally, we only need to prove the Claim. Again using the Tail-integral formula, 
		\begin{equation*}
			\begin{split}
				\mathbb E(Z-a)_+&=\int_{a}^{+\infty}\mathbb{P}(Z>u)du,\\
				\mathbb E(Z-a)_+^2&=2\int_a^\infty(u-a)\mathbb{P}(Z>u)du.
			\end{split}
		\end{equation*}
		And the first derivative
		\begin{equation*}
			\begin{split}
				\frac{d}{da}\mathbb E(Z-a)_+&=-\mathbb{P}(Z>a),\\
				\frac{d}{da}\mathbb E(Z-a)_+^2&=-2\int_a^\infty\mathbb{P}(Z>u)du=-2\mathbb E(Z-a)_+.
			\end{split}
		\end{equation*}
		By the direct calculation, 
		\begin{equation*}
			R'(a)=\frac{(\mathbb E(Z-a)_+)^2-\mathbb{P}(Z>a)\mathbb E(Z-a)_+^2}{\operatorname{Var}(Z-a)_+^{\frac{3}{2}}}.
		\end{equation*}
		We only need to prove that the numerator is less than or equal to $0$, and it can be deduced by Cauchy-Schwarz inequality:
		\begin{equation*}
			\begin{split}
				(\mathbb E(Z-a)_+)^2&=\left(\int_{\mathbb{R}}1_{\{Z>a\}}(Z)(Z-a)_+dT_\nu(Z)\right)^2\\
				&<\int_{\mathbb{R}}1_{\{Z>a\}}(Z)dT_\nu(Z)\int_{\mathbb{R}}(Z-a)_+^2dT_\nu(Z)\\
				&=\mathbb{P}(Z>a)\mathbb E(Z-a)_+^2,
			\end{split}
		\end{equation*}
		and the second inequality is strict because, by the equality characterization of the Cauchy-Schwarz inequality, equality holds iff $1_{\{Z>a\}}=(Z-a)_+$, which also means that $Z$ is a constant. This is a contradiction. Thus we prove that $R$ is strictly decreasing.
		
		Moreover, the range of $R$ is $(0,+\infty)$, Indeed,
		\begin{equation*}
			\begin{split}
				R(a)^2&=\frac{(\mathbb E(Z-a)_+)^2}{\operatorname{Var}(Z-a)_+}\\
				&\leq\frac{\mathbb{P}(Z>a)\mathbb E(z-a)_+^2}{\mathbb E(Z-a)^2_+-\mathbb{P}(Z>a)\mathbb E(Z-a)_+^2}\\
				&=\frac{\mathbb{P}(Z>a)}{\mathbb{P}(Z\leq a)}\rightarrow0
			\end{split}
		\end{equation*}
		as $a\rightarrow+\infty$. The second inequality is from the Cauchy-Schwarz inequality. On the other side, when $a\rightarrow-\infty$, we have
		\begin{equation}\label{L2 convergence}
			(Z-a)_++a=\max(Z,a)\rightarrow Z\quad\text{in}\quad L^2.
		\end{equation}
		This is because, when $Z\geq a$,
		\begin{equation*}
			\max(Z,a)-Z=0,
		\end{equation*}
		and when $Z<a$, and we can of course assume $a<0$, then
		\begin{equation*}
			|\max(Z-a)-Z|^2=(Z-a)^2\leq Z^2.
		\end{equation*}
		Since $Z$ has finite 2 moment from $\nu>2$, by Dominated Convergence Theorem, we have
		\begin{equation*}
			\mathbb E|\max(Z,a),Z|^2\leq \mathbb E\left(Z^21_{\{Z<a\}}\right)\rightarrow0\quad \text{as}\quad a\rightarrow-\infty.
		\end{equation*}
		Since $\mathbb E(Z)=0$ and $\operatorname{Var}(Z)>0$, this directly means that
		\begin{equation*}
			\begin{split}
				\mathbb E(Z-a)_+&\rightarrow-a\rightarrow+\infty,\\
				\operatorname{Var}(Z-a)_+&=\operatorname{Var}\left((Z-a)_++a\right)\rightarrow\operatorname{Var}Z>0.
			\end{split}
		\end{equation*}
		Thus $R(a)\rightarrow+\infty$ as $a\rightarrow-\infty$. By the continuity of the function, the range is indeed $(0,+\infty)$. We prove the claim and we finish the whole proof. 
	\end{proof}
	
	\subsection{The centered-moment bound}\label{sec-moment}
	\begin{proof}[Proof of Theorem \ref{prob-moment}]
		We divide the proof into four steps.
		
		\medskip
		\noindent
		\textbf{Step 1. Reduction to \(Y=g(Z)\) and identification of the median.}
		
		By Theorem \ref{prob-representation}, there exists a finite, non-decreasing convex function
		\(g:\mathbb R\to\mathbb R\) such that $f(X)\overset{\mathrm d}=g(Z).$ By the symmetry of \(Z\) and monotonicity of \(g\),
		\[
		\mathbb P(g(Z)\le g(0))\ge\frac12,
		\qquad
		\mathbb P(g(Z)\ge g(0))\ge\frac12.
		\]
		Thus \(g(0)\) is the median, and the median is unique. Indeed, if \(u<g(0)\), continuity of \(g\)
		provides \(\delta>0\) such that \(g(-\delta)>u\). Consequently,
		\[
		\mathbb P(g(Z)\le u)
		\le T_\nu(-\delta)<\frac{1}{2},
		\]
		which contradicts to the definition of median. The other direction can be discussed similarly. Thus $m=g(0).$
		
		\medskip
		\noindent
		\textbf{Step 2. Choice of \(b>0\) and reduction to a half-line probability.}
		
		If
		\[
		\mathbb P\!\left(
		Y\le m-\frac{\|Y-m\|_q}{\|\max(Z,-r)\|_q}r
		\right)=0,
		\]
		there is nothing to prove. We therefore assume
		that this probability is positive.
		Since \(g\) is continuous and non-decreasing, the set
		\[
		\left\{
		Z\in\mathbb R:
		g(Z)\le m-\frac{\|Y-m\|_q}{\|\max(Z,-r)\|_q}r
		\right\}
		\]
		is a non-empty closed lower half line. Due to the symmetry of the Cauchy-type measure, the definition of the median and $r>0$, its right endpoint
		is therefore strictly negative. Hence there exists \(b>0\) such
		that this set is exactly $(-\infty,-b].$
		And at the endpoint, we have
		\begin{equation}\label{eq:four-step-endpoint}
			g(-b)
			=
			m-\frac{\|Y-m\|_q}{\|\max(Z,-r)\|_q}r.
		\end{equation}
		
		After determining this \(b\), we can immediately that
		\begin{equation}\label{eq:four-step-half-line}
			\begin{aligned}
				\mathbb P\!\left(
				Y\le m-\frac{\|Y-m\|_q}{\|\max(Z,-r)\|_q}r
				\right)=\mathbb P(Z\le-b)=T_\nu(-b).
			\end{aligned}
		\end{equation}
		Since $T_\nu$ is strictly increasing, to prove \eqref{prob-optimal-moment}, we only need to prove
		\[
		b\ge r.
		\]
		
		\medskip
		\noindent
		\textbf{Step 3. The geometric intuition of the convex \(g\).}
		
		\begin{figure}[htbp]
			\centering
			\includegraphics[width=0.6\textwidth]{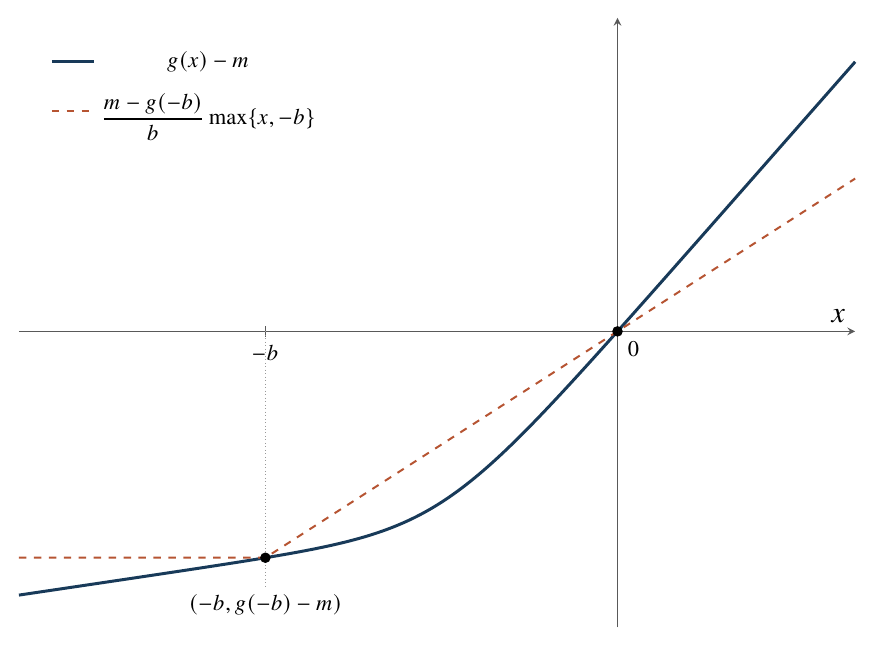}
			\caption{geometric illustration of \ref{eq:four-step-pointwise}}
			\label{fig:myfigure}
		\end{figure}
		In this step, we first fix \(Z\), and we will prove the following geometric inequality:
		\begin{equation}\label{eq:four-step-pointwise}
			|g(x)-m|
			\ge
			\frac{r\|Y-m\|_q}
			{b\|\max(Z,-r)\|_q}
			\,|\max(x,-b)|,
			\qquad x\in\mathbb R,
		\end{equation}
		and its geometric illustration is shown in Figure \ref{fig:myfigure}.
		The coefficient in \eqref{eq:four-step-pointwise} is the slope of
		the chord joining \((-b,g(-b)-m)\) to \((0,0)\):
		\[
		\frac{m-g(-b)}b
		=
		\frac{r\|Y-m\|_q}
		{b\|\max(Z,-r)\|_q}.
		\]
		We prove the comparison on three intervals:
		
		(1). If \(x\le-b\), monotonicity and
		\eqref{eq:four-step-endpoint} give
		\[
		g(x)-m
		\le g(-b)-m
		=
		-\frac{r\|Y-m\|_q}{\|\max(Z,-r)\|_q}.
		\]
		
		(2). If \(-b\le x\le0\), convexity makes the graph of \(g\) below
		the straight line from \((-b,g(-b))\) to \((0,g(0))\). Thus
		\[
		g(x)
		\le
		\frac{-x}{b}g(-b)
		+
		\left(1+\frac xb\right)g(0).
		\]
		Using \(g(0)=m\) and \eqref{eq:four-step-endpoint}, we obtain
		\[
		g(x)-m
		\le
		\frac{r\|Y-m\|_q}
		{b\|\max(Z,-r)\|_q}\,x
		\le0.
		\]
		Taking absolute values proves
		\eqref{eq:four-step-pointwise} on this interval.
		
		(3). If \(x>0\), since $g$ is convex, we deduce that
		\[
		\frac{g(x)-g(0)}x
		\ge
		\frac{g(0)-g(-b)}b
		=
		\frac{r\|Y-m\|_q}
		{b\|\max(Z,-r)\|_q}.
		\]
		Consequently,
		\[
		g(x)-m
		\ge
		\frac{r\|Y-m\|_q}
		{b\|\max(Z,-r)\|_q}\,x>0,
		\]
		which proves \eqref{eq:four-step-pointwise} for \(x>0\).
		
		After finish the proof of \eqref{eq:four-step-pointwise}, let $x=Z$, and we can easily deduce the following inequality by taking the \(q\)-th expectation on both sides:
		\begin{equation}\label{eq:q-th expectation}
			\|Y-m\|_q^q
			\ge
			\left(
			\frac{r\|Y-m\|_q}
			{b\|\max(Z,-r)\|_q}
			\right)^q
			\|\max(Z,-b)\|_q^q,
		\end{equation}
		and since $\|Y-m\|_q>0$, it is equivalent to
		\begin{equation}\label{eq:four-step-ratio}
			\frac{\|\max(Z,-b)\|_q}{b}
			\le
			\frac{\|\max(Z,-r)\|_q}{r}.
		\end{equation}
		
		\medskip
		\noindent
		\textbf{Step 4.
			\(t\mapsto\frac{\|\max(Z,-t)\|_q}{t}\) is strictly decreasing.}
		
		For every \(t>0\), splitting according to the sign of \(Z\)
		gives
		\begin{equation}\label{eq:four-step-monotonicity}
			\left(
			\frac{\|\max(Z,-t)\|_q}{t}
			\right)^q
			=
			\frac{\mathbb EZ_+^q}{t^q}
			+
			\mathbb E\!\left[
			\min\left\{\frac{(-Z)_+}{t},1\right\}^{q}
			\right].
		\end{equation}
		Since
		$0<\mathbb EZ_+^q<+\infty,$
		the first term on the right-hand side of
		\eqref{eq:four-step-monotonicity} is strictly decreasing in \(t\).
		The second term is non-increasing, because its integrand is
		non-increasing in \(t\) for every fixed value of \(Z\).
		Their sum is therefore strictly decreasing. Since \(q>0\), so
		\[
		t\longmapsto\frac{\|\max(Z,-t)\|_q}{t}
		\]
		is strictly increasing.
		
		Finally, applying this strict monotonicity to
		\eqref{eq:four-step-ratio} yields \(b\ge r\).
		Together with \eqref{eq:four-step-half-line}, we can deduce that
		\[
		\mathbb P\!\left(
		Y\le m-\frac{r\|Y-m\|_q}{\|\max(Z,-r)\|_q}
		\right)
		=
		T_\nu(-b)
		\le T_\nu(-r),
		\]
		which is exactly \eqref{prob-optimal-moment}.
		
		To the characterization of equality, suppose that
		\eqref{prob-optimal-moment} is an equality. Since \(T_\nu\) is
		strictly increasing, \eqref{eq:four-step-half-line} implies
		\(b=r\). The inequality \eqref{eq:q-th expectation} in Step 3 must also
		be an equality. This means that the integrands must be equal almost everywhere. Combined with the continuity of the integrands, this implies that the integrands are equal everywhere, that is to say,
		\[
		|g(z)-m|
		=
		\frac{\|Y-m\|_q}{\|\max(Z,-r)\|_q}
		\,|\max(z,-r)|,
		\qquad z\in\mathbb R.
		\]
		And the sign information established in Step 3 can also deduce that
		\[
		g(z)
		=
		m+
		\frac{\|Y-m\|_q}{\|\max(Z,-r)\|_q}
		\max(z,-r),
		\qquad z\in\mathbb R.
		\]
		Conversely, for every function
		\[
		g(z)=m+c\max(z,-r),
		\qquad c>0,
		\]
		Taking it into the left-hand side of \eqref{prob-optimal-moment}, we have
		\begin{equation*}
			\begin{split}
				\mathbb P\left(
				Y\le m-\frac{r\|Y-m\|_q}{\|\max(Z,-r)\|_q}
				\right)&=\mathbb P\left(
				m+c\max(Z,-r)\le m-\frac{rc\|\max(Z,-r)\|_q}{\|\max(Z,-r)\|_q}\right)\\
				&=\mathbb{P}\left(\max(Z,-r)\leq-r\right)\\
				&=\mathbb{P}(Z\leq-r)=T_\nu(-r).
			\end{split}
		\end{equation*}
		Thus the inequality \eqref{prob-optimal-moment} becomes equality. We finish the proof.
	\end{proof}
	
	We would like to show through the following property that the condition $q<\nu$ is necessary. This is because for any $q\geq\nu$, $E(Y-m)_+^q=\infty$.
	\begin{proposition}[The moment threshold]\label{moment-threshold-prop}
		If $f:\R^n\to\R$ is finite, convex and non-constant, then for every $m\in\R$ and every $q\ge\nu$,
		\[
		\mathbb E(f(X)-m)_+^q=\infty.
		\]
		In particular for Cauchy measure ($\nu=1$), it does not have a finite first moment or a finite variance.
	\end{proposition}
	
	\begin{proof}
		Since $f$ is finite convex and non-constant, there exists a point $x_0$ that admits a sub-gradient that is not zero $v$. The supporting affine function gives
		\[
		f(x)\ge f(x_0)+\ip{v}{x-x_0}=\ip{v}{x}+b,\quad \text{for some}\ b\in\mathbb{R}.
		\]
		Since Cauchy-type measure is rotational invariant, $\ip{v}{X}\stackrel{\mathrm d}=|v|Z$. Then for sufficiently large positive $z$, 
		\begin{equation*}
			(|v|z+b-m)_+\ge |v|z/2.
		\end{equation*}
		On the other hand, the Cauchy-type density is bounded below by $Cz^{-\nu-1}$ for some constant $C$. Therefore, for $q\geq\nu$,
		\[
		\mathbb E(f(X)-m)_+^q\ge
		\mathbb E(|v|Z+b-m)_+^q
		\ge C\int_{z_0}^\infty z^{q-\nu-1}\dd z=\infty.
		\]
	\end{proof}
	
	\section{The Gaussian limit}
	\label{sec-gaussian-limit}
	
	As we mentioned earlier in the introduction, the Cauchy-type measure $\mu_{n,\nu}$ tends exactly to the Gaussian measure $\gamma_n$ as $\nu\rightarrow+\infty$ in total variation distance. A natural question is whether Theorems \ref{prob-mean} and \ref{prob-moment} also hold for the Gaussian case. The purpose of this subsection is to extend Theorems \ref{prob-mean} and \ref{prob-moment} to the Gaussian measure, and our approach is by means of approximation; of course, we could directly repeat the proofs of Theorems \ref{prob-mean} and \ref{prob-moment} in the Gaussian setting, but we prefer to adopt the approximation approach in order to reflect the passage from the Cauchy-type measure to the Gaussian measure.
	
	\begin{corollary}[Gaussian limits of Theorems~\ref{prob-mean}
		and~\ref{prob-moment}]\label{gauss-exact}
		Let \(G_n\sim N(0,I_n)\), \(G\sim N(0,1)\), and \(Y=f(G_n)\),
		where \(f:\mathbb R^n\to\mathbb R\) is finite and convex.
		
		\smallskip
		\noindent
		\textit{(i) Mean--variance lower deviations.}
		If \(0<\operatorname{Var}Y<\infty\), then, for every
		\(a\in\mathbb R\),
		\begin{equation}\label{gauss-exact-eq}
			\mathbb P\!\left(
			Y\le \mathbb EY-
			\frac{\mathbb E(G-a)_+}
			{\sqrt{\operatorname{Var}((G-a)_+)}}\sqrt{\operatorname{Var}Y}
			\right)
			\le \Phi(a).
		\end{equation}
		Equality holds if and only if
		\[
		Y\overset{\mathrm d}{=}\beta+c(G-a)_+,
		\qquad \beta\in\mathbb R,\quad c>0.
		\]
		Moreover, fix \(y<\mathbb EY\) such that
		\(\mathbb P(Y\le y)>0\), and define
		\[
		a_y=\Phi^{-1}\bigl(\mathbb P(Y\le y)\bigr),
		\qquad
		Y_*
		=
		y+\frac{\mathbb EY-y}{\mathbb E(G-a_y)_+}(G-a_y)_+.
		\]
		Then
		\begin{equation}\label{gauss-W2}
			W_2^2\bigl(\mathcal L(Y),\mathcal L(Y_*)\bigr)
			\le
			\operatorname{Var}Y-
			\frac{(\mathbb EY-y)^2\operatorname{Var}((G-a_y)_+)}
			{\bigl(\mathbb E(G-a_y)_+\bigr)^2}.
		\end{equation}
		
		\smallskip
		\noindent
		\textit{(ii) Median--moment lower deviations.}
		Let \(q>0\), let \(m\) be the median of \(Y\), and suppose that
		\(0<\|Y-m\|_q<\infty\). Then, for every \(r>0\),
		\begin{equation}\label{gauss-moment-limit}
			\mathbb P\!\left(
			Y\le m-
			\frac{r\|Y-m\|_q}{\|\max(G,-r)\|_q}
			\right)
			\le \Phi(-r).
		\end{equation}
		Equality holds if and only if
		\[
		Y\overset{\mathrm d}{=}\beta+c\max(G,-r),
		\qquad \beta\in\mathbb R,\quad c>0.
		\]
	\end{corollary}
	
	\begin{proof}
		We first establish the assertions under the additional assumption
		that \(f\) is globally Lipschitz, and then remove this assumption
		by convex approximation.
		
		\medskip
		\noindent
		\textbf{Step 1. Globally Lipschitz convex functions.}
		
		Suppose that \(f\) is non-constant, convex, and globally Lipschitz.
		As introduced in the introduction,
		\[
		X_\nu=\sqrt{\frac{\nu}{W_\nu}} G_n\sim\mu_{n,\nu},
		\qquad
		Z_\nu=\sqrt{\frac{\nu}{W_\nu}} G\sim T_\nu,
		\]
		where \(W_\nu\sim\chi_\nu^2\) is independent of \(G_n\) and \(G\). 
		
		We shall use two convergence properties. First, for every fixed
		\(p>0\),
		\[
		\mathbb E\sqrt{\frac{\nu}{W_\nu}}^p
		=
		\left(\frac{\nu}{2}\right)^{p/2}
		\frac{\Gamma((\nu-p)/2)}{\Gamma(\nu/2)}
		\longrightarrow1,
		\qquad \nu\to\infty.
		\]
		Since \(\sqrt{\frac{\nu}{W_\nu}}\to1\) in probability and a moment of any fixed order
		larger than \(p\) is eventually uniformly bounded, uniform integrability gives
		\[
		\mathbb E\left|\sqrt{\frac{\nu}{W_\nu}}-1\right|^p\longrightarrow0.
		\]
		Let \(Y_\nu=f(X_\nu)\). Since $f$ is globally Lipschitz, let $L$ be the Lipschitz's constant. We deduce  
		\begin{equation}\label{gauss-limit:Lp}
			\mathbb E|Y_\nu-Y|^p\longrightarrow0,
			\qquad
			\mathbb E|Z_\nu-G|^p\longrightarrow0,
		\end{equation}
		because of
		\begin{equation*}
			\mathbb{E}|Y_\nu-Y|^p=\mathbb{E}|f(X_\nu)-f(G_n)|^p\leq L^p \mathbb{E}|X_\nu-G_n|^p\longrightarrow0.
		\end{equation*}
		Secondly, the proof of Corollary~\ref{gaussianlimit} gives
		$\mu_{n,\nu}\longrightarrow\gamma_n$
		in total variation. Thus, for every fixed \(y\in\mathbb R\),
		\begin{equation}\label{gauss-limit:fixed-threshold}
			\mathbb P(Y_\nu\le y)\longrightarrow\mathbb P(Y\le y).
		\end{equation}
		Accordingly, we first pass to the limit at a fixed \(y\).
		
		\medskip
		\noindent
		\textbf{Global Lipschitz convex functions for (i).}
		We first prove the stability \eqref{gauss-W2}. Fix \(y<\mathbb EY\) with \(\mathbb P(Y\le y)>0\), and set
		\[
		a_\nu
		=
		T_\nu^{-1}\bigl(\mathbb P(Y_\nu\le y)\bigr).
		\]
		By \eqref{gauss-limit:fixed-threshold}, convergence of the Cauchy-type distribution functions, and continuity and strict monotonicity
		of \(\Phi\),
		\[
		a_\nu\longrightarrow a_y.
		\]
		Consider the corresponding comparison variable in
		Theorem~\ref{prob-mean},
		\[
		Y_{*}^\nu
		=
		y+
		\frac{\mathbb EY_\nu-y}
		{\mathbb E(Z_\nu-a_\nu)_+}(Z_\nu-a_\nu)_+.
		\]
		The estimate
		\[
		\bigl|(Z_\nu-a_\nu)_+-(G-a_y)_+\bigr|
		\le |Z_\nu-G|+|a_\nu-a_y|
		\]
		and \eqref{gauss-limit:Lp} imply
		\[
		(Z_\nu-a_\nu)_+\longrightarrow(G-a_y)_+
		\quad\text{in }L^2.
		\]
		Their means and variances therefore converge, and the limiting
		mean is strictly positive. Together with
		\(\mathbb EY_\nu\to\mathbb EY\), this gives
		\[
		Y_{*}^\nu\longrightarrow Y_*
		\quad\text{in }L^2.
		\]
		Consequently,
		\[
		\mathcal L(Y_\nu)\longrightarrow\mathcal L(Y),
		\qquad
		\mathcal L(Y_{*}^\nu)\longrightarrow\mathcal L(Y_*)
		\quad\text{in }W_2.
		\]
		The triangle inequality for \(W_2\) also gives convergence of
		the distance between the two laws. Applying the stability estimate
		in Theorem~\ref{prob-mean} to \(Y_\nu\) and letting
		\(\nu\to\infty\), we obtain \eqref{gauss-W2}.
		
		We next prove \eqref{gauss-exact-eq}. The right-hand side of \eqref{gauss-W2} is
		nonnegative, so
		\begin{equation}\label{gauss-limit:fixed-variance}
			\frac{\mathbb EY-y}{\sqrt{\operatorname{Var}Y}}
			\le
			\frac{\mathbb E(G-a_y)_+}
			{\sqrt{\operatorname{Var}((G-a_y)_+)}}.
		\end{equation}
		Now choose \(y=\mathbb EY-
		\frac{\mathbb E(G-a)_+}
		{\sqrt{\operatorname{Var}((G-a)_+)}}\sqrt{\operatorname{Var}Y}\) . If \(\mathbb P(Y\le y)=0\), there is nothing to prove. If \(\mathbb P(Y\le y)>0\), \eqref{gauss-limit:fixed-variance} gives
		\[
		\frac{\mathbb E(G-a)_+}
		{\sqrt{\operatorname{Var}((G-a)_+)}}
		\le
		\frac{\mathbb E(G-a_y)_+}
		{\sqrt{\operatorname{Var}((G-a_y)_+)}}.
		\]
		Since the strict decrease is true, which implies \(a_y\le a\). Hence
		\[
		\mathbb P(Y\le y)=\Phi(a_y)\le\Phi(a),
		\]
		which proves \eqref{gauss-exact-eq}.
		
		\medskip
		\noindent
		\textbf{Global Lipschitz convex functions for (ii).}
		Fix \(q>0\), and let \(m_\nu\) and \(m\) be the medians of
		\(Y_\nu\) and \(Y\), respectively. We first show that
		\begin{equation}\label{gauss-limit:medians}
			m_\nu\longrightarrow m.
		\end{equation}
		A finite non-constant convex function is continuous and has an
		interval as its range. Under an everywhere positive density,
		every nonempty open subinterval of this range has positive
		probability. Its distribution therefore has a unique median.
		In particular, for every \(\varepsilon>0\),
		\[
		\mathbb P(Y\le m-\varepsilon)<\frac12,
		\qquad
		\mathbb P(Y\ge m+\varepsilon)<\frac12.
		\]
		Weak convergence and the Portmanteau theorem imply that these
		strict inequalities also hold with \(Y_\nu\) in place of \(Y\)
		for all sufficiently large \(\nu\). Thus
		\[
		m-\varepsilon<m_\nu<m+\varepsilon,
		\]
		proving \eqref{gauss-limit:medians}.
		Combining this, \eqref{gauss-limit:Lp} and Minkowski's inequality (since $\|\cdot\|_q$ is already a norm when $q\geq1$) we obtain
		\begin{equation*}
			|\|Y_\nu-m_\nu\|_q-\|Y-m\|_q|\leq \|Y_\nu-Y\|_q+\|m_\nu-m\|_q\longrightarrow0,
		\end{equation*}
		for \(q\ge1\). And we use
		\[
		\bigl||u|^q-|v|^q\bigr|\le |u-v|^q,
		\]
		when $0<q<1$ to get the same thing.
		Thus for all $q>0$,
		\begin{equation}\label{gauss-limit:centered-moments}
			\|Y_\nu-m_\nu\|_q\longrightarrow\|Y-m\|_q.
		\end{equation}
		
		Fix \(y<m\) with \(\mathbb P(Y\le y)>0\), and put
		\[
		b_\nu
		=
		-T_\nu^{-1}\bigl(\mathbb P(Y_\nu\le y)\bigr),
		\qquad
		b
		=
		-\Phi^{-1}\bigl(\mathbb P(Y\le y)\bigr)>0.
		\]
		By \eqref{gauss-limit:fixed-threshold}, \(b_\nu\to b\).
		For sufficiently large \(\nu\), we have \(q<\nu\) and
		\(y<m_\nu\). The proof of
		Theorem~\ref{prob-moment} gives
		\[
		\frac{\|\max(Z_\nu,-b_\nu)\|_q}{b_\nu}
		\le
		\frac{\|Y_\nu-m_\nu\|_q}{m_\nu-y}.
		\]
		Since
		\[
		|\max(Z_\nu,-b_\nu)-\max(G,-b)|
		\le |Z_\nu-G|+|b_\nu-b|,
		\]
		we may use \eqref{gauss-limit:Lp}--%
		\eqref{gauss-limit:centered-moments} to pass to the limit.
		This yields
		\begin{equation}\label{gauss-limit:fixed-moment}
			\frac{\|\max(G,-b)\|_q}{b}
			\le
			\frac{\|Y-m\|_q}{m-y}.
		\end{equation}
		Choose \(y\) to be the threshold in
		\eqref{gauss-moment-limit}, that is $y=m-
		\frac{r\|Y-m\|_q}{\|\max(G,-r)\|_q}$. If the corresponding probability
		is zero, there is nothing to prove. Otherwise,
		\eqref{gauss-limit:fixed-moment} becomes
		\[
		\frac{\|\max(G,-b)\|_q}{b}
		\le
		\frac{\|\max(G,-r)\|_q}{r}.
		\]
		The strict decrease proved in Step~4 of the proof of
		Theorem~\ref{prob-moment} applies equally to the Gaussian
		reference variable. It follows that \(b\ge r\), and hence
		\[
		\mathbb P(Y\le y)=\Phi(-b)\le\Phi(-r).
		\]
		This proves \eqref{gauss-moment-limit}.
		
		\medskip
		\noindent
		\textbf{Step 2. General finite convex functions.}
		
		We now assume only that \(f\) is finite, nonconstant, and convex,
		with the Gaussian integrability required in the relevant assertion.
		
		Choose a supporting affine function \(\ell_0\) of \(f\) with
		nonzero slope. For \(L\ge|\nabla\ell_0|\), define
		\[
		f_L(x)
		=
		\sup\left\{
		\ell(x):
		\ell\text{ is affine},\
		\ell\le f,\
		|\nabla\ell|\le L
		\right\}.
		\]
		Then \(f_L\) is finite, convex, non-constant, and globally
		\(L\)-Lipschitz. Moreover,
		\begin{equation}\label{gauss-limit:convex-approximation}
			\ell_0\le f_L\le f,
			\qquad
			f_L\longrightarrow f\quad\text{as}\ L\rightarrow\infty.
		\end{equation}
		To verify the last assertion, fix \(x\) and choose a supporting
		affine function of \(f\) at \(x\). Once \(L\) exceeds the magnitude
		of its slope, this function is included in the supremum, and
		therefore \(f_L(x)=f(x)\).
		
		Write \(Y_L=f_L(G_n)\). For every fixed \(y\),
		\begin{equation}\label{gauss-limit:monotone-threshold}
			\{Y_L\le y\}\longrightarrow\{Y\le y\},
			\qquad
			\mathbb P(Y_L\le y)\longrightarrow\mathbb P(Y\le y),\quad \text{as}\ L\rightarrow\infty.
		\end{equation}
		Since an affine Gaussian variable has moments of every positive
		order, Dominated Convergence Theorem shows that, whenever
		\(\mathbb E|Y|^p<\infty\),
		\begin{equation}\label{gauss-limit:approximation-Lp}
			\mathbb E|Y_L-Y|^p\longrightarrow0.
		\end{equation}
		
		\medskip
		\noindent
		\textbf{Finite convex function for (i).}
		Suppose that \(0<\operatorname{Var}Y<\infty\).
		By \eqref{gauss-limit:approximation-Lp},
		\[
		Y_L\longrightarrow Y
		\quad\text{in }L^2,
		\]
		so the means and variances converge.
		
		Fix \(y<\mathbb EY\) with \(\mathbb P(Y\le y)>0\), and set
		\[
		a_L=\Phi^{-1}\bigl(\mathbb P(Y_L\le y)\bigr).
		\]
		By \eqref{gauss-limit:monotone-threshold}, \(a_L\to a_y\).
		The corresponding comparison variables
		\[
		Y_{L,*}
		=
		y+
		\frac{\mathbb EY_L-y}{\mathbb E(G-a_L)_+}(G-a_L)_+
		\]
		also converge to \(Y_*\) in \(L^2\).
		
		Each \(f_L\) belongs to the class treated in Part~1.
		For sufficiently large \(L\), we may therefore apply
		\eqref{gauss-W2} to \(Y_L\). Letting \(L\to\infty\), the same
		Wasserstein convergence argument as in Step~1.1 establishes
		\eqref{gauss-W2} for \(Y\). Equation
		\eqref{gauss-limit:fixed-variance} and strict monotonicity of the
		reference moment ratio then give \eqref{gauss-exact-eq}.
		
		For equality, let \(y\) be the threshold in
		\eqref{gauss-exact-eq}. If equality holds there, then \(a_y=a\).
		The right-hand side of \eqref{gauss-W2} is consequently zero,
		which forces
		\[
		\mathcal L(Y)=\mathcal L(Y_*).
		\]
		Thus
		\[
		Y\overset{\mathrm d}{=}\beta+c(G-a)_+,
		\qquad c>0.
		\]
		Conversely, for this distribution, the threshold in
		\eqref{gauss-exact-eq} is exactly \(\beta\), and the lower-tail
		probability is \(\Phi(a)\). Hence equality holds.
		
		\medskip
		\noindent
		\textbf{Finite convex functions for (ii).}
		Fix \(q>0\), suppose that \(0<\|Y-m\|_q<\infty\), and let
		\(m_L\) be the median of \(Y_L\). By
		\eqref{gauss-limit:approximation-Lp} and the median convergence
		argument in Step~1.2,
		\begin{equation}\label{gauss-limit:approximation-medians}
			m_L\longrightarrow m,
			\qquad
			\|Y_L-m_L\|_q\longrightarrow\|Y-m\|_q.
		\end{equation}
		
		Fix \(y<m\) with \(\mathbb P(Y\le y)>0\), and define
		\[
		b_L
		=
		-\Phi^{-1}\bigl(\mathbb P(Y_L\le y)\bigr),
		\qquad
		b
		=
		-\Phi^{-1}\bigl(\mathbb P(Y\le y)\bigr).
		\]
		By \eqref{gauss-limit:monotone-threshold}, \(b_L\to b>0\).
		For sufficiently large \(L\), the fixed-threshold estimate
		from Step~1.2 applies to \(Y_L\), giving
		\[
		\frac{\|\max(G,-b_L)\|_q}{b_L}
		\le
		\frac{\|Y_L-m_L\|_q}{m_L-y}.
		\]
		Using \eqref{gauss-limit:approximation-medians} and letting
		\(L\to\infty\), we obtain
		\eqref{gauss-limit:fixed-moment} for the general convex function.
		Choosing \(y\) to be the threshold in
		\eqref{gauss-moment-limit} and repeating the monotonicity
		argument in Step~1.2 proves \eqref{gauss-moment-limit}.
		
		It remains to characterize equality. The Gaussian Ehrhard
		inequality already obtained in Corollary~\ref{gaussianlimit},
		together with the inverse-distribution construction used in
		Theorem~\ref{prob-representation}, gives
		\[
		Y\overset{\mathrm d}{=}g(G),
		\]
		where \(g\) is nondecreasing and convex, and \(m=g(0)\).
		If equality holds in \eqref{gauss-moment-limit}, the corresponding
		sublevel half-line has endpoint \(-r\). Hence
		\[
		g(-r)
		=
		m-\frac{r\|Y-m\|_q}{\|\max(G,-r)\|_q}.
		\]
		Applying the chord comparison from the proof of
		Theorem~\ref{prob-moment} gives
		\begin{equation}\label{gauss-limit:equality-envelope}
			|g(x)-m|
			\ge
			\frac{\|Y-m\|_q}{\|\max(G,-r)\|_q}
			|\max(x,-r)|,
			\qquad x\in\mathbb R.
		\end{equation}
		The \(q\)-th powers of the two sides have equal Gaussian
		integrals. Their difference is continuous and nonnegative;
		since the Gaussian density is everywhere positive, this
		difference vanishes identically. Using the sign information
		in the chord comparison, we conclude that
		\[
		g(x)
		=
		m+
		\frac{\|Y-m\|_q}{\|\max(G,-r)\|_q}
		\max(x,-r),
		\qquad x\in\mathbb R.
		\]
		This proves necessity.
		
		Conversely, if
		\[
		Y\overset{\mathrm d}{=}\beta+c\max(G,-r),
		\qquad c>0,
		\]
		then its median is \(\beta\), its centered \(q\)-moment is
		\(c\|\max(G,-r)\|_q\), and the event in
		\eqref{gauss-moment-limit} is exactly \(\{G\le-r\}\).
		Thus equality holds.
	\end{proof}
	
	\section{Comparison with the Gaussian bound and the upper-tail obstruction}\label{sec-comparison}
	
	We prove in the following proposition that, \eqref{gauss-exact-eq} is stronger than \eqref{pv-background} and \eqref{intro-valettas-variance} when $t>1$, and \eqref{gauss-exact-eq} is stronger than \eqref{intro-chen} when $t>0$. In fact, \eqref{gauss-exact-eq} have already attained the optimal lower-tail bound as shown before.
	
	\begin{proposition}[Comparison with the explicit variance bound]\label{gauss-comparison}
		For $t>1$, taking $a_t$ the unique solution of
		\begin{equation*}
			\frac{\mathbb E(G-a_t)_+}{\sqrt{\operatorname{Var}(G-a_t)_+}}=t,
		\end{equation*}
		and the existence and uniqueness is from the strict monotonicity of $R_G$ in Corollary \ref{gauss-exact}. Then
		\begin{equation}
			\label{eq:comparison-valettas}
			\Phi(a_t)<\Phi\!\left(-\frac{t-1}{\sqrt{2\pi}}\right)<\frac{1}{2}e^{-\frac{\pi(t-1)^2}{1024}},
			\qquad t>1,
		\end{equation}
		and
		\begin{equation}
			\label{eq:comparison-chen}
			\Phi(a_t)<e^{-t^2/2},
			\qquad t>0.
		\end{equation}
	\end{proposition}
	
	\begin{proof}
		We first prove \eqref{eq:comparison-valettas}. Since $R_G(0)=\frac{1}{\sqrt{\pi-1}}<1$, the equation $R_G(a)=t>1$ shows that $a<0$. Since
		\begin{equation*}
			(G-a)_+\leq G_+-a,
		\end{equation*}
		by taking expectations on both sides, we have
		\begin{equation}
			\label{comparison on expectation}
			\mathbb{E}(G-a)_+\leq\frac{1}{\sqrt{2\pi}}-a<\frac{1}{2}-a.
		\end{equation}
		After using the Tail-integral formula (Please see a similar calculation in \textbf{Claim} in the proof of Theorem \ref{prob-mean}), 
		\begin{equation*}
			\begin{split}
				\frac{d}{da}\operatorname{Var}(G-a)_+&=\frac{d}{da}\left(\mathbb{E}(G-a)_+^2-(\mathbb{E}(G-a)_+)^2\right)\\
				&=-2\mathbb{E}(G-a)_++2\mathbb{E}(G-a)_+\mathbb{P}(G>a)\\
				&=-2\mathbb{E}(G-a)_+\Phi(a)<0.
			\end{split}
		\end{equation*}
		Consequently, 
		\begin{equation}\label{comparison on var}
			\operatorname{Var}(G-a)_+\geq\operatorname{Var}G_+=\frac{1}{2}-\frac{1}{2\pi}>\frac{1}{4}.
		\end{equation}
		Combining \eqref{comparison on expectation} and \eqref{comparison on var}, it deduces that
		\begin{equation*}
			t=\frac{\mathbb E(G-a_t)_+}{\sqrt{\operatorname{Var}(G-a_t)_+}}<\frac{\frac{1}{2}-a}{\sqrt{\frac{1}{4}}}=1-2a.
		\end{equation*}
		For any $x\ge0$, we have
		\[
		\Phi(-x)=\frac{1}{\sqrt{2\pi}}\int_{-\infty}^{-x} e^{-\frac{t^2}{2}}dt=\frac{e^{-\frac{x^2}{2}}}{\sqrt{2\pi}}
		\int_0^\infty e^{-\frac{u^2}{2}-xu}\dd u\le\frac12e^{-\frac{x^2}{2}},
		\]
		by the change of variables $u=-t-x$. Thus,
		\begin{equation*}
			\begin{split}
				\Phi(a)&<\Phi\left(-\frac{t-1}{2}\right)\\
				&<\Phi\left(-\frac{t-1}{\sqrt{2\pi}}\right)\\
				&\leq\frac{1}{2}e^{-\frac{(t-1)^2}{4\pi}}<\frac{1}{2}e^{-\frac{\pi(t-1)^2}{1024}},
			\end{split}
		\end{equation*}
		and the final inequality is because $\frac{1}{4\pi}\geq\frac{\pi}{1024}$.
		
		We next prove \eqref{eq:comparison-chen}. Fix $t>0$ and set $H=(G-a)_+$.
		Chen's negative-parameter moment generating function estimate
		\cite[Theorem~2(1)]{Chen26} gives
		\[
		\mathbb E e^{\lambda(H-\mathbb E H)}
		\le
		e^{\frac{\lambda^2}{2}\operatorname{Var}(H)},
		\qquad \lambda\le0.
		\]
		Choose
		\[
		\lambda=-\frac{\mathbb E H}{\operatorname{Var}(H)}<0.
		\]
		Since $H\ge0$ and $\mathbb P(H>0)>0$, we have
		\[
		\mathbb P(H=0)=\mathbb{E}[e^{\lambda H}1_{\{H=0\}}]<\mathbb{E}[e^{\lambda H}1_{\{H=0\}}]+\mathbb{E}[e^{\lambda H}1_{\{H>0\}}]=\mathbb E e^{\lambda H}.
		\]
		Therefore,
		\begin{equation*}
			\begin{split}
				\Phi(a)
				&=\mathbb P(H=0)\\
				&<\mathbb E e^{\lambda H}=\mathbb{E}e^{\lambda(H-\mathbb{E}H)}\cdot e^{\lambda\mathbb{E}H}\\
				&\le
				e^{\frac{\lambda^2}{2}\operatorname{Var}(H)+\lambda\mathbb E H}\\
				&=
				e^{-\frac{(\mathbb E H)^2}{2\operatorname{Var}H}}
				=e^{-\frac{t^2}{2}}. 
			\end{split}
		\end{equation*}
		This proves \eqref{eq:comparison-chen}.
	\end{proof}
	
	The comparison concerns the explicit mean--variance estimate quoted above. Valettas \cite[Theorem~2.13]{Valettas17} obtained a sharp constant for a different normalization, involving the median and a positive-part first moment. Corollary~\ref{gauss-exact} determines the whole mean--variance curve.
	
	\begin{proposition}[Upper deviations cannot have a uniform faster order]\label{upper-obstruction}
		There is no function $B(t)=o(t^{-2})$ as $t\to\infty$ such that
		\[
		\mathbb P(Y\ge\mathbb EY+t\sqrt{\operatorname{Var}Y})\le B(t)
		\]
		for every non-constant convex Gaussian observable $Y$ of finite variance. The obstruction already occurs among one-dimensional positive-part affine functions.
	\end{proposition}
	
	\begin{proof}
		For $a>0$, let $Y_a=(G-a)_+$. Write $\phi$ for the standard Gaussian density. Substitution $G=a+\frac{u}{a}$ gives, for $j=1,2$,
		\begin{equation*}
			\begin{split}
				\mathbb EY_a^j&=\int_a^{+\infty}(x-a)^jd\gamma_1(x)=\frac{\phi(a)}{a^{j+1}}
				\int_0^\infty u^je^{-u-\frac{u^2}{2a^2}}\dd u.
			\end{split}
		\end{equation*}
		Let $a\rightarrow+\infty$. By Dominated Convergence theorem with the dominator $u^je^{-u}$, it deduces that
		\[
		\mathbb{E}Y_a\sim\frac{\phi(a)}{a^2},\qquad
		\operatorname{Var}Y_a\sim\frac{2\phi(a)}{a^3}\qquad a\rightarrow+\infty,
		\]
		here $\sim$ means equivalent infinitesimals. The same substitution gives
		\[
		\mathbb P\left(Y_a\ge\frac{1}{a}\right)=\frac{\phi(a)}a
		\int_1^\infty e^{-u-\frac{u^2}{2a^2}}\dd u
		\sim\frac{e^{-1}\phi(a)}a.
		\]
		Let 
		$$t_a=\frac{a^{-1}-\mathbb{E}Y_a}{\sqrt{\operatorname{Var}Y_a}}\rightarrow+\infty\qquad\text{as}\quad a\rightarrow+\infty.$$ 
		Thus,
		\begin{equation*}
			\begin{split}
				t_a^2\mathbb P(Y_a\ge m_a+t_a\sqrt{v_a})&=t_a^2\mathbb{P}\left(Y_a\geq\frac{1}{a}\right)\\
				&=\frac{(1-a\mathbb{E}Y_a)^2}{a^2\operatorname{Var}Y_a}\mathbb{P}\left(G\geq a+\frac{1}{a}\right)\\ 
				&\longrightarrow\frac{a}{2\phi(a)}\cdot\frac{\phi(a)}{ea}=\frac1{2e}>0.
			\end{split}
		\end{equation*}
		This contradicts any proposed uniform bound $B(t)=o(t^{-2})$.
	\end{proof}
	Conversely, Cantelli's inequality gives a uniform upper bound of order $t^{-2}$. Indeed, the Cantelli's inequality
	\[
	\mathbb{P}\left(Y\geq\mathbb{E}Y+t\sqrt{\operatorname{Var}Y}\right)\leq\frac{1}{1+t^2}.
	\]
	Together with Proposition~\ref{upper-obstruction}, this identifies the best possible uniform decay power (but without asserting optimality of Cantelli's constant within the convex Gaussian class).

\end{document}